\documentclass[journal]{IEEEtran}

\usepackage[hidelinks]{hyperref}
\usepackage{cite}
\usepackage{amsmath,amssymb,amsfonts,amsthm}
\usepackage{algorithmic}
\usepackage{graphicx}
\usepackage{algorithm,algorithmic}
\usepackage{textcomp}
\usepackage{mathtools}
\usepackage{cleveref}
\usepackage{pgfplots}
\pgfplotsset{compat=1.18}
\usepackage{subcaption}

\usepackage{dblfloatfix}
\usepackage{booktabs} 
\usepackage{enumerate}

\usepgfplotslibrary{groupplots}

\newtheorem{theorem}{Theorem}[section]
\newtheorem{lemma}[theorem]{Lemma}

\newtheorem{remark}[theorem]{Remark}

\newtheorem{ass}{Assumption}

\newcommand{\R}{\mathbb{R}}
\newcommand{\N}{\mathbb{N}}

\newcommand\setD{\mathcal{D}}
\newcommand\setC{\mathcal{C}}
\newcommand\setS{\mathcal{S}}
\newcommand{\Hy}{\mathcal{H}}

\newcommand\one{{\mathbf{1}}}
\newcommand\Metz{\mathcal{M}}
\newcommand\K{\mathcal{K}}
\newcommand\KL{\mathcal{KL}}
\newcommand\setA{\mathcal{A}}

\newcommand\D{\mathrm d}

\newcommand\J{\mathrm j}
\newcommand\abs[1]{\left\lvert#1\right\rvert}
\newcommand\norm[1]{\left\lVert#1\right\rVert}
\newcommand\sigpow[2]{\lceil #1 \rfloor^{#2}}
\newcommand\inner[2]{\left\langle #1 ,\, #2 \right\rangle}

\DeclareMathOperator\dom{dom}
\DeclareMathOperator\sign{sign}

\DeclareMathOperator\diag{diag}
\DeclareMathOperator*\argmin{arg\, min}

\definecolor{QualPurple}{HTML}{A559AA}
\definecolor{QualTeal}{HTML}{59A89C}

\allowdisplaybreaks

\def\BibTeX{{\rm B\kern-.05em{\sc i\kern-.025em b}\kern-.08em
    T\kern-.1667em\lower.7ex\hbox{E}\kern-.125emX}}

\begin{document}
\title{
Practical Stabilization of Switched Affine Systems under Dwell Time: From Polytopic to \texorpdfstring{$p$}{p}-Norm Lyapunov Functions
%
%
%
}

\author{Felipe Cinto, Alexis J. Vallarella, Hernan Haimovich, Paulo C. Pellanda
\thanks{This study was financed in part by the Coordenação de Aperfeiçoamento de Pessoal de Nível Superior - Brasil (CAPES) - Finance Code 001, and by Agencia I+D+i, Argentina, under grant PICT 2021-I-A-0730.}
\thanks{F. Cinto, A. J. Vallarella and H. Haimovich are with the Centro Internacional Franco-Argentino de Cs. de la Informaci\'on y de Sistemas (CIFASIS), UNR--CONICET, Argentina (e-mails: \{cinto, vallarella, haimovich\}@cifasis-conicet.gov.ar).}
\thanks{Paulo César Pellanda is with the
Departments of Defense Engineering and Electrical Engineering, Military Institute of Engineering, Rio de Janeiro 22290-270, Brazil (e-mail: pellanda@ime.eb.br).}
}

\maketitle

\begin{abstract}
This paper addresses the stabilization of switched affine systems under dwell-time constraints around a desired operating point that is not a common equilibrium of the subsystems. Since asymptotic stabilization is generally unattainable in this setting, the objective is to ensure that trajectories converge to and remain within a prescribed region containing the desired operating point. The main contribution is the development of Lyapunov-based tools that accommodate target regions more general than those induced by quadratic Lyapunov functions.
In particular, we establish a constructive Lyapunov characterization for a broad family of weighted $p$-norm functions, which generalizes the classical quadratic Lyapunov equivalence and underpins the proposed hybrid-control and dwell-time developments.
Focusing on the case in which the desired region is a polytope,
and building on existing Lyapunov-based methods, we develop a constructive two-step method that starts from a polytopic Lyapunov certificate for the average system and yields a smooth weighted $p$-norm Lyapunov function for all sufficiently large $p$. We prove that this extension guarantees practical asymptotic stability of the desired operating point and yields a Lyapunov-certified invariant region under minimum dwell-time switching. Numerical examples illustrate the effectiveness of the proposed method.
\end{abstract}

\begin{IEEEkeywords}
Switched systems,
Dwell-time switching,
Stability of hybrid systems,
Nonlinear systems
\end{IEEEkeywords}


\section{Introduction} \label{sec:intro}

\IEEEPARstart{S}{witched} systems provide a modeling framework for dynamical systems that exhibit both continuous dynamics and discrete transitions \cite{liberzon2003switched}. 
Such systems consist of multiple subsystems, with a switching signal selecting which subsystem is active at each time.
In the design of switching strategies, the literature primarily distinguishes between two approaches: \emph{arbitrarily fast switching} \cite{bolzern2004switchedaffine}, in which the active subsystem can change without restrictions on timing, and \emph{dwell-time switching} \cite{briat2015convex}, which imposes constraints on the time interval between consecutive switches.

Stabilization of an equilibrium point has been extensively investigated for switched linear systems (SLS), in which all subsystems share the origin as a common equilibrium point \cite{cheng2005stabilization,geromel2006stabilitystabilization,geromel2008feedback,antsaklis2009survey}.
By contrast, for switched affine systems, the origin need not be an equilibrium point of any subsystem, and asymptotic stabilization can be achieved only through arbitrarily fast switching.
In many practical applications, however, dwell-time constraints limit the switching frequency \cite{russoACC, cinto2025switching}.
In such cases, only practical stabilization is achievable, in the sense that the state converges to and remains within a desired neighborhood of the desired point. 

Switched systems subject to dwell-time constraints have attracted substantial research attention because they capture practical switching limitations. For example,
\cite{briat2015convex} provides convex conditions for robust stabilization of uncertain switched systems under mode-dependent dwell-time; 
\cite{Makarenkov201889} presents closed-form dwell-time formulas that confine oscillations to prescribed neighborhoods of equilibria in planar switched affine systems, with direct applications to non-spiking neuron models; and 
\cite{albea2019practical} uses space or time regularization within a hybrid systems framework to achieve practical stabilization of operating points in switched affine systems while guaranteeing minimum dwell-times.
More recently, 
\cite{Yin2023} investigates set convergence and practical stability for switched systems with distinct equilibria and mixed stable-unstable modes by means of multiple Lyapunov-like functions outside compact sets;
\cite{DiFerdinando20241691} provides a methodology for sampled-data control of nonlinear asynchronous systems using steepest descent switching feedback to ensure semiglobal practical stability under arbitrary dwell-times;
\cite{Ghawash2025} proposes a real-time multitask learning framework to approximate model predictive control solutions for switched affine systems subject to dwell-time constraints, thereby facilitating embedded implementation; and
\cite{Russo2025} proposes switching laws relying on Lyapunov-Metzler inequalities to stabilize switched linear and affine systems while guaranteeing an upper bound on a quadratic performance cost.
Most of these references require, in one way or another, the use of quadratic Lyapunov functions.

Despite the well-known advantages of quadratic Lyapunov functions, more general function classes may be required or advantageous in specific applications. Notable alternatives include polytopic \cite{blanchini1995nonquadratic} and higher-degree polynomial Lyapunov functions \cite{Parrilo2003}. In practical stabilization problems, target regions are often specified through suitable bounds on the individual state variables, thereby defining a polytopic region in the state space, rather than by an ellipsoidal region that combines all variables simultaneously. 
Since the Lyapunov-based strategy of \cite{albea2019practical} yields convergence regions as sublevel sets of a quadratic Lyapunov function, matching the geometry of these sublevel sets to polytopic target regions motivates the consideration of polytopic Lyapunov functions.
However, the lack of differentiability of polytopic functions introduces important challenges. In particular, \cite{cinto2024polytopic} shows that replacing a quadratic Lyapunov function with a polytopic one, while leaving the remaining conceptual framework unchanged, may fail to guarantee stability when the switching strategy selects the subsystem yielding the steepest decrease of the Lyapunov function, as in \cite{albea2019practical}. 

This work addresses this limitation by considering a family of $p$-norm Lyapunov functions that interpolates between the quadratic and polytopic cases: the quadratic case is recovered for $p=2$, whereas the polytopic case is approached for large values of $p$.
Our main contribution is the generalization of the steepest-decrease switching strategy of \cite{albea2019practical} to this broader class of Lyapunov functions. 
On the theoretical side, we establish a characterization of the Hurwitz property in terms of weighted $p$-norm Lyapunov functions, thereby extending the classical quadratic Lyapunov equivalence to the whole family $p\in(1,\infty)$; the proof is constructive and based on Jordan decomposition. On the practical side, 
starting from a polytopic Lyapunov certificate for the average system, we develop a constructive two-step procedure to obtain a smooth weighted $p$-norm Lyapunov function that can be used directly in the proposed hybrid switching strategies under unconstrained and dwell-time switching.
This generalization offers practical advantages, such as allowing a larger minimum dwell time for a given target region or yielding a smaller practical stability region for a prescribed minimum dwell-time.

The remainder of this paper is organized as follows.
Section~\ref{sec:hybrid} introduces the weighted $p$-norm Lyapunov functions, establishes a generalized Lyapunov equivalence, and develops the hybrid switching strategy, proving uniform global asymptotic stability (UGAS) of the origin under arbitrarily fast switching.
Section~\ref{sec:dwell_time} adapts the space and time regularization techniques of \cite{albea2019practical} to this setting, establishing practical UGAS under minimum dwell-time constraints.
Given a polytopic function that characterizes the desired target region,  Section~\ref{sec:construction} provides a procedure to certify that this function is a Lyapunov function, and then to approximate it by a $p$-norm Lyapunov function for a sufficiently large $p$.
Numerical examples are presented in Section~\ref{sec:examples}, illustrating that increasing $p$ beyond the quadratic case yields practical advantages, such as larger minimum dwell times for a given target region, or smaller practical stability regions for a prescribed minimum dwell time. Section \ref{sec:conclusion} concludes.


{\small
\textbf{Notation.}
Let $\R$, $\R_{\geq 0}$, and $\N$ denote the sets of real, nonnegative real, and natural numbers, respectively, and let $\N_N \coloneqq \{1,\dots,N\}$ for $N \in \N$.
For $x, y \in \R^n$, $x^\top$ denotes transpose, and  $\langle x, y \rangle \coloneqq x^\top y$ the inner product.
The symbol $\one$ denotes a vector of appropriate dimension with all entries equal to $1$.
Absolute value, $p$-norm, and infinity norm are denoted by $\abs{\boldsymbol{\cdot}}$, $\norm{\boldsymbol{\cdot}}_p$, and $\norm{\boldsymbol{\cdot}}_\infty$, respectively.
For a closed set $\setA \subset \R^n$, define $\norm{x}_\setA \coloneqq \inf_{y \in \setA} \norm{x-y}_2$.
For $x\in\R$, define the signed-power function $\sigpow{x}{p} \coloneqq \abs{x}^p \sign(x)$; for $x\in\R^n$, define $\sigpow{x}{p}$ componentwise by $\sigpow{x}{p} \coloneqq [\sigpow{x_1}{p} \dots \sigpow{x_n}{p}]^\top$.
For $X,Y \in \R^{r \times n}$, $X_{i,j}$ denotes the $(i,j)$-th entry, $X_{i:}$ the $i$-th row, and $X \preceq Y$ denotes componentwise inequality, i.e., $X_{i,j} \leq Y_{i,j}, \,\, \forall i,j$. The norm of $X$ induced by the $p$-norm is denoted by $\norm{X}_p$, and its Moore--Penrose pseudoinverse by $X^\dag$.
The identity matrix of order $k$ is denoted by $I_k$.
$\mathcal{N}_k$ denotes the nilpotent matrix of order $k$ with ones on the superdiagonal and zeros elsewhere, i.e., $[\mathcal{N}_k]_{i,j} = 1$ if $j = i+1$ and $0$ otherwise.
The operator $\diag(M_1,\dots,M_k)$ denotes the block-diagonal matrix with diagonal blocks $M_1,\dots,M_k$.
For $A \in \R^{m \times n}$ and $B \in \R^{p \times q}$, $A \otimes B \in \R^{mp \times nq}$ denotes the Kronecker product.
Given $\Gamma \in \R^{r \times r}$, its associated Metzler matrix\footnote{A matrix is Metzler if its off-diagonal entries are nonnegative.}
$\Metz(\Gamma) \in \R^{r \times r}$ is defined by $\Metz(\Gamma)_{i,k} = \Gamma_{i,k}$ if $i = k$, and $\Metz(\Gamma)_{i,k} = |\Gamma_{i,k}|$ otherwise.
For a square matrix $M \in \R^{r \times r}$, the matrix measure associated with the $p$-norm (also called logarithmic norm \cite{soderlind2006logarithmic}) is defined as
\begin{equation}
    \mu_p(M) \coloneqq  \lim_{h\to0^+} \frac{\norm{I+hM}_p-1}{h}.
\end{equation}
A function $\alpha: \R_{\geq 0} \to \R_{\geq 0}$ is a class $\K_\infty$ function ($\alpha\in\K_\infty$) if it is continuous, strictly increasing, unbounded, and $\alpha(0)=0$.
A function $\beta: \R_{\geq 0} \times \R_{\geq 0} \to \R_{\geq 0}$ is a class $\KL$ function ($\beta\in\KL$) if it is nondecreasing in its first argument, nonincreasing in its second argument, $\lim_{r\to0^+} \beta(r, s) = 0$ for each $s \in \R_{\geq0}$, and $\lim_{s\to\infty} \beta(r, s) = 0$ for each $r \in \R_{\geq0}$.
}


\section{Hybrid Control via Weighted \texorpdfstring{$p$}{p}-Norms}
\label{sec:hybrid}

\subsection{Problem Statement}
Consider the switched affine system
\begin{equation}
    \label{eq:sas}
    \dot x(t) = A_{\sigma(t)} x(t) + b_{\sigma(t)} \eqqcolon f_{\sigma(t)}(x),
\end{equation}
where $x \in \R^n$ is the state, $\sigma : \R_{\geq0} \to \N_N$ is the switching signal, and $A_i \in \R^{n \times n}$, $b_i \in \R^{n}$ characterize subsystem $i \in \N_N$.
We focus on the stability of the origin, since the analysis for an arbitrary operating point follows by a suitable change of variables.  

Define the simplex
\begin{equation}
\label{eq:unit_simplex}
    \Lambda \coloneqq \left\{ \lambda \in \R^N : \lambda_i \in [0,1] \ \forall i\in\N_N, \ \lambda^\top \one = 1 \right\},
\end{equation}
and, for each $\lambda \in \Lambda$, define the convex combinations
\begin{equation}
    \label{eq:Alam_blam}
    A_\lambda \coloneqq \sum_{i=1}^N \lambda_i A_i, \quad
    b_\lambda \coloneqq \sum_{i=1}^N \lambda_i b_i.
\end{equation}
We rely on the following standard assumption for the stabilization of switched affine systems \cite{bolzern2004switchedaffine,albea2019practical,cinto2024polytopic}.

\medskip
\begin{ass}
\label{ass:1}
There exists $\lambda^* \in \Lambda$ such that $A_{\lambda^*}$ is Hurwitz and $b_{\lambda^*} = 0$.
\end{ass}
\medskip

Under Assumption~\ref{ass:1}, the average system
\begin{equation}
    \label{eq:average_sys}
    \dot x(t) = A_{\lambda^*} x(t)
\end{equation}
is globally asymptotically stable.
The objective of this work is to develop a framework for the design of a state-dependent switching strategy $\tilde\sigma(x(t)) = \sigma(t)$, based on weighted $p$-norm Lyapunov functions for \eqref{eq:average_sys}, that guarantees UGAS of the origin (in the sense of \cite{goebel2012hybrid}) under unconstrained switching, i.e., arbitrarily fast switching, and practical UGAS under dwell-time constraints.


\subsection{Weighted p-norm Lyapunov Functions}

Given a full column rank matrix $R \in \R^{r \times n}$ and $p\in(1,\infty)$, consider the weighted $p$-norm function
\begin{equation}
    V_{R,p}(x) \coloneqq \frac{1}{p} \norm{Rx}_p^p = \frac{1}{p} \sum_{i \in \N_r} \abs{R_{i:} x}^p.
     \label{eq:p-normlyapfunction}
\end{equation}
Note that $V_{R,2} = \frac{1}{2} x^\top P x$, where $P=R^\top R$, is the classical quadratic Lyapunov function.
For any finite $p>1$, the function $V_{R,p}(x)$ is differentiable on $\R^n$, with the gradient given by the column vector
\begin{equation}
     \nabla V_{R,p}(x) =  \sum_{i\in\N_r} \sigpow{R_{i:} x}{p-1} R_{i:}^\top = R^\top \sigpow{R x}{p-1} .  \label{eq:gradient}
\end{equation}
To achieve the objective stated above, we seek a weighted $p$-norm Lyapunov function for the average system \eqref{eq:average_sys}.
The following theorem characterizes the Hurwitz property of a matrix $A$ in terms of the existence of such a Lyapunov function, thereby generalizing the classical Lyapunov equation.


\medskip
\begin{theorem}
\label{th:Lyap}
Let $A \in \R^{n \times n}$ and $p \in (1,\infty)$. Then $A$ is Hurwitz if and only if there exist $r,s \in \N$, and full column rank matrices $R \in \R^{r \times n}$ and $S \in \R^{s \times n}$ such that
\begin{equation}
    \label{eq:p-norm-lyap-cross}
    \inner{\nabla V_{R,p}(x)}{Ax} + V_{S,p}(x) \leq 0
    \quad \text{for all } x \in \R^n.
\end{equation}
\end{theorem}
\medskip

\begin{proof}
The proof relies on standard Jordan decomposition techniques; constructive details are given in Appendix~\ref{app:Lyap}. 

\medskip
\noindent
\emph{Direct implication.}
Assume that $A$ is Hurwitz. Consider the system $\dot x = A x$. By \cite[Th. 3.4.1.5]{horn2012matrix}, there exists a nonsingular matrix $T\in\R^{n\times n}$ such that $A = T J T^{-1}$, where $J$ is the real Jordan form
\begin{align}
    J = \text{diag} &\left(C_{n_1}(\alpha_1, \omega_1), \dots, C_{n_q}(\alpha_q, \omega_q), \right. \nonumber\\
    & \; \left. J_{n_{q+1}}(s_{q+1}), \dots, J_{n_{q+\ell}}(s_{q+\ell})\right),
\end{align}
with $C_{n_i}(\alpha_i,\omega_i)$ and $J_{n_{q+j}}(s_{q+j})$ defined in \eqref{eq:block_complex} and \eqref{eq:block_real} for all $i \in \N_q$ and $j \in \N_\ell$, respectively.
The blocks $C_{n_i}(\alpha_i, \omega_i)$ correspond to the pairs of complex conjugate eigenvalues $-\alpha_i \pm \J\, \omega_i$ of $A$ with multiplicity $n_i$, whereas the blocks $J_{n_{q+j}}(s_{q+j})$ correspond to the real eigenvalues $s_{q+j}$ of $A$ with multiplicity $n_{q+j}$. 

Define $z \coloneqq T^{-1}x$, so that $\dot z = T^{-1} \dot x = J z$.
Consistently with the block structure of $J$, partition $z$ as
\begin{equation}
    z = \begin{bmatrix} z_1^\top,\,  \dots\,,\,  z_q^\top,\,  z_{q+1}^\top,\, \dots\,,\,  z_{q+\ell}^\top \end{bmatrix}^\top,
\end{equation}
where $z_i \in \R^{2n_i}$ for $i\in \N_q$ and $z_{q+j}\in \R^{n_{q+j}}$ for $j \in \N_{\ell}$. 
The dynamics then decouple into $q+\ell$ independent subsystems:
\begin{subequations}
\begin{align}
    \dot z_i &= C_{n_i}(\alpha_i, \omega_i) z_i, && \forall i \in \N_q, \label{eq:complex_sub} \\
    \dot z_{q+j} &= J_{n_{q+j}}(s_{q+j}) z_{q+j}, && \forall j \in \N_\ell.
    \label{eq:real_sub}
\end{align}
\end{subequations}

Since $A$ is Hurwitz, we have $\alpha_i > 0$ for all $i \in \N_q$ and $s_{q+j} < 0$ for all $j \in \N_\ell$.
Applying Lemma~\ref{lem:Jordan_Complex_Block} to  \eqref{eq:complex_sub} and Lemma~\ref{lem:Jordan_Real_Block} to \eqref{eq:real_sub}, it follows that, for every $k \in \N_{q+\ell}$, there exist full column rank matrices $R_k$ and positive scalars $\beta_k$ such that $\inner{\nabla V_{R_k,p}(z_k)}{\dot z_k} \leq -\beta_k V_{R_k,p}(z_k)$.
Define $\hat{R} \coloneqq \diag(R_1, \dots, R_{q+\ell})$. 
Applying Lemma~\ref{lemma:Lyap_Aggregation}, it follows that 
$V_{\hat R, p}(z) = \frac{1}{p}\|\hat{R} z\|_p^p$ satisfies 
$\langle\nabla V_{\hat R,p}(z),\,J z\rangle \leq -\beta V_{\hat R, p}(z)$,  $\beta \coloneqq \min_{k\in\N_{q+\ell}} \beta_k > 0$.

Define $R \coloneqq \hat{R} T^{-1}$.
Since $\hat{R}$ has full column rank and $T$ is nonsingular, $R$ also has full column rank.
Moreover, $V_{R,p}(x) = V_{\hat{R}, p}(z)$, and therefore
\begin{equation}
    \inner{\nabla V_{R,p}(x)}{Ax} = \langle{\nabla V_{\hat R,p}(z)},\,{J z}\rangle \leq -\beta V_{R,p}(x). \label{eq:dotVR}
\end{equation}
Let $S \coloneqq \beta^\frac{1}{p} R$. Then, $S$ has full column rank and, by homogeneity of the $p$-norm, 
\begin{equation}
    V_{S,p}(x) = \frac{1}{p} \norm{\beta^{1/p} R x}_p^p = \frac{\beta}{p} \norm{Rx}_p^p = \beta V_{R,p}(x). \label{eq:V_S}
\end{equation}
Substituting \eqref{eq:V_S} into \eqref{eq:dotVR} yields \eqref{eq:p-norm-lyap-cross}.

\medskip
\noindent
\emph{Converse implication.}
Since $R$ and $S$ have full column rank, $V_{R,p}$ and $V_{S,p}$ are positive definite and radially unbounded.
Consider the system $\dot x = A x$.
From \eqref{eq:p-norm-lyap-cross}, 
\begin{equation}
    \tfrac{d}{dt} V_{R,p}(x) =  \inner{\nabla V_{R,p}(x)}{Ax} \leq -V_{S,p}(x), \;\; \forall x\in\R^n.
\end{equation} 
Global asymptotic stability then follows from Lyapunov's direct method.
Since, for linear systems, global asymptotic stability of $\dot x = Ax$ is equivalent to $A$ being Hurwitz \cite[Th.~4.5]{khalil2002nonlinear}, the proof is complete.
\end{proof}

\medskip
\begin{remark}
The proof of Theorem~\ref{th:Lyap} is constructive. In particular, the results in Appendix~\ref{app:Lyap} yield an explicit procedure, based on the Jordan decomposition of $A$, for computing a matrix $R$ that is independent of $p$.
Once $R$ is fixed, the matrix $S = \beta^{1/p} R$ satisfies \eqref{eq:p-norm-lyap-cross} for any $p \in (1, \infty)$.
\end{remark}


\subsection{Hybrid Control Strategy}
\label{subsec:hybrid_control}

Given a Lyapunov function for the average system \eqref{eq:average_sys}, the steepest-decrease switching strategy of \cite{albea2019practical} relies on the property that, for every $x \neq 0$, there exists at least one subsystem for which the directional derivative of that function along the corresponding vector field is strictly negative. 
This property is guaranteed for quadratic Lyapunov functions \cite{albea2019practical}, but may fail for polytopic ($\infty$-norm) Lyapunov functions \cite{cinto2024polytopic}.
The following lemma shows that it also holds for any weighted $p$-norm Lyapunov function for the average system \eqref{eq:average_sys}.

\medskip
\begin{lemma}
\label{lem:min_k}
Suppose Assumption~\ref{ass:1} holds. Let $p \in (1,\infty)$, and let $R\in\R^{r\times n}$ and $S\in\R^{s\times n}$ be full column rank matrices such that \eqref{eq:p-norm-lyap-cross} holds with 
$A = A_{\lambda^*}$.
Then, for all $x \in \R^n \setminus \{0\}$, 
\begin{equation}
    \underset{k\in\N_N}{\min} \inner{\nabla V_{R,p}(x)}{f_k (x)} \leq - V_{S,p}(x).
\end{equation}
\end{lemma}

\medskip
\begin{proof} For any set of scalars $\{z_k\}_{k \in \N_N}$,
\[
\min_{k \in \N_N} z_k
=
\min_{\lambda \in \Lambda} \sum_{k \in \N_N} \lambda_k z_k,
\]
since the minimum of a finite set coincides with the minimum over all its convex combinations. Therefore,
\begin{align*}
&\min_{k \in \N_N} \inner{\nabla V_{R,p}(x)}{f_k(x)}  = \min_{\lambda \in \Lambda} \sum_{k \in \mathbb{N}_N} \! \lambda_k \inner{\nabla V_{R,p}(x)}{f_k(x)} \notag\\
&\quad = \min_{\lambda \in \Lambda} \inner{ \nabla V_{R,p}(x) }{A_\lambda x + b_\lambda}  \leq \inner{\nabla V_{R,p}(x)}{A_{\lambda^*}x + b_{\lambda^*}} \\ &\quad = \inner{\nabla V_{R,p}(x)}{A_{\lambda^*} x}.
\end{align*}
Since \eqref{eq:p-norm-lyap-cross} holds for $A=A_{\lambda^*}$, it follows that $\inner{\nabla V_{R,p}(x)}{A_{\lambda^*}x} \le -V_{S,p}(x)$,
which proves the result.
\end{proof}
\medskip

Since Lemma~\ref{lem:min_k} establishes the required property for weighted $p$-norm Lyapunov functions, the steepest-decrease strategy of \cite{albea2019practical} extends naturally to this class. 
Let $p\in(1,\infty)$, and let $R$ and $S$ be full column rank matrices such that \eqref{eq:p-norm-lyap-cross} holds with $A=A_{\lambda^*}$; under Assumption~\ref{ass:1}, their existence is guaranteed by Theorem~\ref{th:Lyap}.
In particular, $V_{R,p}$ is a Lyapunov function for the average system \eqref{eq:average_sys}.
Adopting the hybrid-systems framework and stability notions of \cite{goebel2012hybrid}, we define the closed-loop hybrid system $\Hy \coloneqq(\setC, f, \setD, G)$ by 
\begin{equation}
\label{eq:hybridsys}    
\Hy: \left\{
\begin{aligned}
\begin{bmatrix}
    \dot x \\ \dot \sigma
\end{bmatrix}
 &= f(x,\sigma), && (x,\sigma) \in \setC, 
 \\[5pt]
\begin{bmatrix}
     x^{+} \\  \sigma^+
\end{bmatrix} &\in G(x,\sigma), && (x,\sigma) \in \setD. \\
\end{aligned}
\right.
\end{equation}
The flow and jump maps are
\begin{subequations}
\label{eq:maps}
\begin{align}
f(x,\sigma) &\!\coloneqq\! \begin{bmatrix}
                f_\sigma (x) \\[1pt] 0
            \end{bmatrix}, \label{eq:maps:f}\\
G(x,\sigma) &\!\coloneqq\! \left\{ \begin{bmatrix}
    x \\ i
\end{bmatrix} : 
    i \in 
   \underset{k \in \N_N}{\argmin} \, \inner{\nabla V_{R,p}(x)}{f_k (x)} 
\right\}, \label{eq:maps:g}
\end{align}
\end{subequations}
and the flow and jump sets are 
\begin{subequations}
\label{eq:sets}
    \begin{align}
\setC\coloneqq\left \{ (x,\sigma): \inner{\nabla V_{R,p}(x)}{f_\sigma (x)} \leq - \eta V_{S,p}(x) \right\}, \label{eq:setC}\\
\setD\coloneqq\left \{ (x,\sigma): \inner{\nabla V_{R,p}(x)}{f_\sigma (x)} \geq - \eta V_{S,p}(x) \right\}, \label{eq:setD}
\end{align}
\end{subequations}
where $\eta \in (0,1)$ is a design parameter. By Lemma~\ref{lem:min_k}, the jump map selects a mode for which the Lyapunov decrease condition is strictly stronger than the flow condition, so that post-jump states belong to the interior of $\setC$ whenever $x \neq 0$.
The resulting hybrid strategy renders the origin UGAS, that is, the set $\{0\}\times\N_N$ is UGAS for $\Hy$.

\medskip
\begin{theorem}
\label{th:UGAS}
Suppose Assumption~\ref{ass:1} holds. Let $p \in (1,\infty)$, and let $R\in\R^{r\times n}$ and $S\in\R^{s\times n}$ be full column rank matrices such that \eqref{eq:p-norm-lyap-cross} holds with $A=A_{\lambda^*}$.
Then the set $\setA \coloneqq  \{0\} \times \N_N$ is UGAS for the hybrid system \eqref{eq:hybridsys}. 
\end{theorem}
\medskip
\begin{proof}
We establish UGAS by applying \cite[Theorem 1]{prieur2014}.
Let $z \coloneqq (x,\, \sigma) \in \R^n \times \N_N$ denote the hybrid state and define 
$W(z) \coloneqq V_{R,p}(x)$.
The distance from $z$ to the closed set $\setA$ is $\norm{z}_\setA = \norm{x}_2$.
Since $R$ has full column rank, $V_{R,p}$ is positive definite and radially unbounded.
Hence, by equivalence of norms, 
there exist $\alpha_1$, $\alpha_2 \in \K_\infty$ such that $\alpha_1(\norm{x}_2)\leq V_{R,p}(x) \leq \alpha_2(\norm{x}_2), \; \forall x\in\R^n$.
Therefore, $\alpha_1(\norm{z}_\setA)\leq W(z) \leq \alpha_2(\norm{z}_\setA), \; \forall z \in \R^n \times \N_N$. 
Now consider $z \in \setC$.
From \eqref{eq:maps:f} and \eqref{eq:setC}, it follows that
\begin{align}
\inner{\nabla W(z)}{f(z)} = \inner{\nabla V_{R,p}(x)}{f_\sigma (x)} \leq - \eta V_{S,p}(x).
\end{align}
Since $S$ has full column rank, $V_{S,p}$ is also positive definite and 
radially unbounded. Thus, there exists $\alpha_3 \in \K_\infty$ such that
$\inner{\nabla W(z)}{f(z)}  \leq - \alpha_3(\norm{x}_2) = - \alpha_3(\norm{z}_\setA)$.

For $z \in \setD$, during jumps, we have $x^+ = x$, and therefore 
$W(z^+) - W(z) = V_{R,p}(x^+) - V_{R,p}(x) = 0$.
Hence, $W$ does not increase across jumps and decreases along flows outside $\setA$. To conclude UGAS by \cite[Theorem 1]{prieur2014}, it remains to verify the semiglobal practical persistent flow property of $\Hy$ (see \cite[Property 1]{prieur2014}), which excludes Zeno behavior and eventually discrete solutions.

Let $\delta, \Delta>0$, and define
$S_{\delta,\Delta} \coloneqq  \{ (x, \sigma): \delta \leq \|x\| \leq \Delta \}$,
together with the restricted hybrid system
\begin{equation}
\label{eq:Hdelta}
\Hy_{\delta,\Delta}\coloneqq  (\setC_{\delta,\Delta},f,\setD_{\delta,\Delta}, G),
\end{equation}
where $\setC_{\delta,\Delta} \coloneqq  \setC \cap S_{\delta,\Delta}$ and 
$\setD_{\delta,\Delta} \coloneqq  \setD \cap S_{\delta,\Delta}$.
Note that $x\neq 0$ for all $(x,\sigma)\in S_{\delta,\Delta}$.

Consider a jump from $z \in\setD_{\delta,\Delta}$. By the definition of $G$ in \eqref{eq:maps:g} and Lemma~\ref{lem:min_k}, after the jump we have
\begin{equation}
    \inner{\nabla W(z^+)}{f(z^+)} \leq - V_{S,p}(x^+) < -\eta V_{S,p}(x^+),
\end{equation}
where we used $\eta \,{\in} (0,1)$ and $V_{S,p}(x^+) > 0$, since $x^+ = x \neq 0$ for all $x\in \setS_{\delta,\Delta}$. 
Hence, $z^+\in\setC_{\delta,\Delta}\backslash\setD_{\delta,\Delta}$.
Given that $\inner{\nabla V_{R,p}(x)}{f_k(x)}$ is continuous for all $(x, k) \in \setS_{\delta,\Delta} \times \N_N$, the sets $\setC_{\delta,\Delta}$ and $\setD_{\delta,\Delta}$ are closed. 
Therefore, by continuity, any solution to $\Hy_{\delta,\Delta}$ starting from $z^+\in\setC_{\delta,\Delta}\backslash\setD_{\delta,\Delta}$, if it exists, must flow for a positive amount of time before reaching $\setD_{\delta,\Delta}$ again. 
Since $\setC_{\delta,\Delta}$ is compact, continuity of solutions with respect to initial conditions \cite[Proposition 6.14]{goebel2012hybrid} implies the existence of a uniform dwell time $\rho(\delta,\Delta)$ between any two consecutive jumps.
This dwell time satisfies \cite[eq. 4]{prieur2014} with the class $\K_\infty$ function
$\gamma(j) = \rho(\delta,\Delta) j$ and $N = \rho(\delta,\Delta)$. Hence, the assumptions of \cite[Theorem 1]{prieur2014} hold, and $\setA$ is UGpAS.  

To conclude UGAS, it remains to prove completeness of maximal solutions.
The hybrid basic conditions in \cite[Assumption 6.5]{goebel2012hybrid} are satisfied: $\setC$ and $\setD$ are closed, $f$ is continuous and locally bounded, and $G$ is outer semicontinuous\footnote{As defined in \cite[Def. 5.9]{goebel2012hybrid}. See Appendix~\ref{app:g_semicont} for the proof.} 
and locally bounded.
Moreover, for any $z \in \setC\cup\setD$, either $z \in \setD$ or $\inner{\nabla V_{R,p}(x)}{f_\sigma(x)}<-\eta V_{S,p}(x)$, in which case $z$ lies in the interior of $\setC$. 
Hence, any solution starting from $\setC\cup\setD$ can be extended either by jumping or by flowing.
Since trajectories cannot escape to infinity in finite time under affine dynamics, and since $\setC\cup\setD = \R^n \times \N_N$, \cite[Proposition 6.10]{goebel2012hybrid} implies that all maximal solutions to $\Hy$ are complete. Therefore, $\setA$ is UGAS for $\Hy$.
\end{proof}

\begin{remark}
    Combining Theorems~\ref{th:Lyap} and \ref{th:UGAS}, Assumption~\ref{ass:1} is sufficient to guarantee, for any $p \in (1,\infty)$, the existence of a $p$-norm-based switching law for which the set $\setA=\{0\}\times\N_N$ is UGAS. 
\end{remark}
\begin{remark}
    The property established in Lemma~\ref{lem:min_k} is not guaranteed for polytopic Lyapunov functions of the average system \cite{cinto2024polytopic}. Therefore, Theorem~\ref{th:UGAS} does not extend directly to that setting without further assumptions.
\end{remark}


\section{Practical Stabilization under Dwell-Time}
\label{sec:dwell_time}

While the hybrid strategy \eqref{eq:hybridsys} guarantees UGAS of the set $\setA=\{0\}\times\N_N$, it may require arbitrarily fast switching as the continuous state approaches the origin. 
To obtain an implementable controller under physical switching constraints, we relax the objective to practical stability, namely, that solutions converge to and remain within a neighborhood of the origin, whose size decreases as the minimum dwell-time is reduced. 
In this section, we adapt the two regularization techniques introduced in \cite{albea2019practical} to the proposed hybrid controller: space regularization, which inhibits switching when the state is sufficiently close to the origin, and time regularization, which enforces a minimum dwell time through an auxiliary timer state.
In both cases, the practical convergence regions are characterized as sublevel sets of $V_{R,p}$.


\subsection{Space Regularization}
\label{subsec:space_regularization}

Space regularization prevents arbitrarily fast switching by inhibiting jumps when the state is sufficiently close to the origin. 
For a given parameter $\varepsilon > 0$, define the space-regularized hybrid system $\mathcal{H}_\varepsilon \coloneqq (\mathcal{C}_\varepsilon, f, \mathcal{D}_\varepsilon, G)$ by modifying the flow and jump sets in \eqref{eq:sets} as
\begin{subequations}
\begin{align}
    \setC_\varepsilon &\coloneqq \setC\, \cup \{(x,\sigma): V_{R,p}(x) \leq \varepsilon\}, \label{eq:sets_eps:C} \\
    \setD_\varepsilon &\coloneqq \setD \cap \{(x,\sigma): V_{R,p}(x) \geq \varepsilon\}. 
\end{align}
\end{subequations}
This modification forces the system to flow in a neighborhood of the origin.
The following theorem summarizes the properties of this regularization; the proof is given in Appendix~\ref{app:thm:reg}.

\medskip
\begin{theorem}
\label{thm:eps}
Suppose Assumption~\ref{ass:1} holds. Let $p \in (1,\infty)$, and let $R\in\R^{r\times n}$ and  $S\in\R^{s\times n}$ be full column rank matrices such that \eqref{eq:p-norm-lyap-cross} holds  with $A=A_{\lambda^*}$. Then, for any $\varepsilon > 0$, the following statements hold:
\begin{enumerate}[a)]
    \item there exists a constant $T(\varepsilon) > 0$ 
    such that all solutions $\phi$ to $\Hy_\varepsilon$ satisfy $t\geq (j-1)T(\varepsilon)$ for all $(t,j)\in\dom\phi$,  \label{thm:eps:dwell}
    \item the set $\setA_\varepsilon \coloneqq \{ (x,\sigma) \in \R^n \times \N_N : V_{R,p}(x) \leq \varepsilon \}$ is UGAS for $\mathcal{H}_\varepsilon$. \label{thm:eps:UGAS}
\end{enumerate}
\end{theorem}

\medskip
\begin{remark}
    Item \ref{thm:eps:dwell}) guarantees a minimum dwell-time $T(\varepsilon)$ between consecutive switches, while still allowing an initial jump at some $t\in[0,T(\varepsilon))$.
\end{remark}
\begin{remark}
    Since $\setC_{\varepsilon_1} \subseteq \setC_{\varepsilon_2}$ for any $\varepsilon_1 \leq \varepsilon_2$ (see~\eqref{eq:sets_eps:C}), increasing $\varepsilon$ enlarges the flow set and, consequently, increases the minimum dwell-time $T(\varepsilon)$, at the expense of enlarging the convergence set $\setA_\varepsilon$.
\end{remark}


\subsection{Time Regularization}
\label{subsec:time_regularization}

An alternative way to prevent arbitrarily fast switching is to enforce a minimum dwell time explicitly between consecutive jumps.
To this end, augment the state with an auxiliary timer $\tau \in \R_{\geq 0}$ and define 
$r\!\left(\frac{\tau}{T}\right) \coloneqq \min\left\{1,\,2-\frac{\tau}{T}\right\}$.
The time-regularized hybrid system $\Hy_T$ is then given by
\begin{equation}
\label{eq:maps_time}
\Hy_T:  \left\{
\begin{aligned}
    &\begin{aligned}
        \begin{bmatrix}
            \dot x \\ \dot \sigma 
        \end{bmatrix}
         &= f(x, \sigma), \\
         \dot \tau &= r\!\left(\frac{\tau}{T}\right),
    \end{aligned} 
    && (x,\sigma,\tau) \in \setC_T \\[10pt]
    &\begin{aligned}
        \begin{bmatrix}
             x^{+} \\  \sigma^+
        \end{bmatrix} &\in G(x,\sigma), \\
             \tau^+ &= 0.
    \end{aligned}
    && (x,\sigma,\tau) \in \setD_T
    \end{aligned}
\right.
\end{equation}
The function $r$ satisfies
$ r\!\left(\frac{\tau}{T}\right)=1 \; \text{for } \tau \in [0,T],\;
r\!\left(\frac{\tau}{T}\right)=2-\frac{\tau}{T} \; \text{for } \tau \in [T,2T]$.
Hence, $\tau$ increases at unit rate until it reaches $T$, and then converges asymptotically to $2T$, thereby remaining in the compact interval $[0,2T]$.
The flow and jump sets are redefined as
\begin{subequations}
\label{eq:sets_time}
\begin{align}
    \setC_T &\coloneqq \left( \setC \times [0, 2T] \right) \cup \left( \R^n \times \N_N \times [0, T] \right), \label{eq:setC_T} \\
    \setD_T &\coloneqq \setD \times [T, 2T]. \label{eq:setD_T}
\end{align}
\end{subequations}
Thus, after each jump, the system is forced to flow for at least $T$ time units, regardless of the Lyapunov derivative condition. 
Once $\tau \geq T$, the original flow and jump conditions are recovered, so that the system continues to flow according to $\setC$ and may jump again through $\setD_T$.
Although this forced flow may produce a transient increase in the  Lyapunov function, the affine dynamics ensure that such an increase remains bounded over any interval of length $T$. 
Moreover, after each jump the subsystem selected by $G$ ensures a strictly negative Lyapunov derivative at the jump instant; hence, by continuity, $T$ can be chosen small enough so that the Lyapunov derivative remains negative over the entire forced-flow interval following every jump. Consequently, the only phase in which $V_{R,p}$ may grow is the initial forced-flow interval preceding the first jump.
The following theorem summarizes the resulting properties of $\Hy_T$; the proof is given in Appendix~\ref{app:thm:reg}.
\medskip
\begin{theorem}
\label{thm:time_reg}
Assume Assumption~\ref{ass:1} holds. Let $p \in (1,\infty)$, and let $R\in\R^{r\times n}$ and $S\in\R^{s\times n}$ be full column rank matrices such that \eqref{eq:p-norm-lyap-cross} holds with $A=A_{\lambda^*}$. Then there exist $T_M >0$, a strictly increasing function $\varepsilon\!:\!(0,T_M)\! \to\!\R_{>0}$ satisfying $\lim_{T\to0^+}\!\varepsilon(T)=0$, a function $\beta\in\KL$, and a function $\alpha\in\K_\infty$ such that for every $T \in (0,T_M )$ and every solution $\phi$ to $\Hy_T$ with initial condition $\phi(0,0)=(x_0,\sigma_0,\tau_0)$, the following statements hold for all $(t,j)\in\dom\phi$:
\begin{enumerate}[a)]
    \item \vspace{0.1cm}
        $\norm{\phi(t,j)}_{\setA_{T}} \leq \beta\!\left(\norm{\phi(0,0)}_{\setA_{T}} \! +  \alpha(\Delta_0 ),\, t+j\right)$, \label{thm:time_reg:UGAS}
        \vspace{0.1cm}
    \item $t\geq (j-1)T$, \label{thm:time_reg:dwell}
\end{enumerate}
where $\setA_{T} {\coloneqq} \{ (x,\sigma,\tau) {\in} \R^n {\times} \N_N {\times} [0,2T] : V_{R,p}(x) \leq \varepsilon(T) \}$, and $\Delta_0 {\coloneqq} \max\{0,\, T {-} \, \tau_0\}$.
\end{theorem}
\medskip

\begin{remark}
When $\tau_0 < T$, the system is forced to flow for up to $\Delta_0$ time units under the initial subsystem $\sigma_0$, which may increase $V_{R,p}$ and cause trajectories starting in $\setA_{T}$ to leave this set temporarily.
Consequently, $\setA_{T}$ is not UGAS for $\Hy_T$. 
The term $\alpha(\Delta_0)$ in item~\ref{thm:time_reg:UGAS}) accounts for 
this overshoot, which is a consequence of the time regularization rather than an artifact of the proof.
Since $\alpha(0)=0$, choosing $\tau_0 \geq T$ recovers a standard $\mathcal{KL}$ bound.
\end{remark}


\section{Constructing the \texorpdfstring{$p$}{p}-Norm Lyapunov Function}
\label{sec:construction}

The results in Section~\ref{sec:dwell_time} guarantee practical UGAS under dwell-time constraints for any valid $p$-norm Lyapunov function $V_{R,p}$ for the average system, with convergence regions given by sublevel sets of $V_{R,p}$.
In practice, operating regions for switched affine systems are often defined by linear state bounds \cite{blanchini2015settheoretic}.
Given a matrix $R$ such that the sublevel sets of $\Psi_{R,\infty}(x)\coloneqq\norm{Rx}_\infty$ describe the desired operating region, if $\Psi_{R,\infty}$ is a Lyapunov function for the average system \eqref{eq:average_sys}, we show how to obtain a corresponding weighted $p$-norm Lyapunov function $V_{R,p}$ for sufficiently large $p$.

\subsection{A Similarity Condition and Matrix Measures}
The proposed method relies on a sufficient condition that links the dynamics matrix $A$ to the Lyapunov matrix $R$ through an auxiliary matrix $\Gamma$. The matrix measure $\mu_p(\Gamma)$ then serves as a certificate ensuring that $V_{R,p}$ satisfies \eqref{eq:p-norm-lyap-cross}. 

\medskip
\begin{lemma}
\label{lem:Gamma}
Let $A\in\R^{n\times n}$ and let $R\in\R^{r\times n}$ be full column rank.
If there exist $\Gamma \in \R^{r \times r}$ and $p \in (1,\infty)$ such that
\begin{equation}
 \Gamma R = R A, \label{eq:GRRA}
\end{equation}
and $\mu_p(\Gamma) < 0$,
then \eqref{eq:p-norm-lyap-cross} holds with
\begin{equation}
   S \coloneqq \left(p\abs{\mu_p(\Gamma)}\right)^\frac{1}{p} R.
\end{equation}
\end{lemma}
\medskip

\begin{proof}
For $x=0$, both sides of \eqref{eq:p-norm-lyap-cross} vanish, so the result holds trivially.
Let $x\neq 0$ and define $y\coloneqq Rx$. Since $R$ has full column rank, we have $y\neq 0$.
Using the expression for $\nabla V_{R,p}$ and \eqref{eq:GRRA}, we obtain
\begin{align}
    \inner{\nabla V_{R,p}(x)}{Ax} \!&= \inner{R^\top \sigpow{Rx}{p-1}}{Ax} \!=\! \inner{\sigpow{Rx}{p-1}}{RAx} \nonumber\\
    &=\inner{\sigpow{Rx}{p-1}\!}{\Gamma Rx} = \inner{\sigpow{y}{p-1}\!}{\Gamma y}. \label{eq:inner_eq}
\end{align}
For $p \in (1,\infty)$, the $p$-norm is continuously differentiable on $\R^r\setminus\{0\}$, with gradient $\nabla \norm{y}_p = \norm{y}_p^{1-p}\,\sigpow{y}{p-1}$.
Substituting this into \eqref{eq:inner_eq} yields
\begin{equation}
    \inner{\nabla V_{R,p}(x)}{Ax} = \norm{y}_p^{p-1} \langle{\nabla \norm{y}_p},\,{\Gamma y}\rangle. \label{eq:inner_norm}
\end{equation}
Since $\norm{\boldsymbol{\cdot}}_p$ is differentiable at $y$, its directional derivative in the direction $\Gamma y$ satisfies
\begin{align}
    &\langle{\nabla \norm{y}_p},\,{\Gamma y}\rangle = \lim_{h\to0^+} \frac{\norm{y+h \, \Gamma y}_p - \norm{y}_p}{h} \nonumber\\
    &\qquad\leq \lim_{h\to0^+} \frac{\norm{I+h \Gamma}_p - 1}{h} \norm{y}_p = \mu_p(\Gamma) \norm{y}_p, \label{eq:dir_deriv}
\end{align}
where we used the submultiplicativity of the induced $p$-norm and the definition of $\mu_p$.

Using \eqref{eq:dir_deriv} in \eqref{eq:inner_norm}, and recalling that $\mu_p(\Gamma)<0$, we obtain
\begin{align} 
&\inner{\nabla V_{R,p}(x)}{Ax} \leq \mu_p(\Gamma) \norm{y}_p^p = \mu_p(\Gamma) \norm{Rx}_p^p \nonumber\\
&\qquad \qquad = - \frac{1}{p} \norm{\left(p\abs{\mu_p(\Gamma)}\right)^\frac{1}{p} Rx}_p^p = - V_{S,p}(x), 
\end{align}
which proves the result.
\end{proof}
\medskip

\begin{remark}
The conditions of Lemma~\ref{lem:Gamma} are sufficient, but not necessary in general, for $V_{R,p}$ to satisfy \eqref{eq:p-norm-lyap-cross}. 
We refer the reader to~\cite{loskot1998further} for a characterization 
of the cases in which necessity also holds.
\end{remark}
\medskip

Given that $R$ has full column rank, any matrix $\Gamma$ satisfying \eqref{eq:GRRA} can be parameterized as $\Gamma = RAR^\dag + Z(I-RR^\dag)$, 
where $Z \in \mathbb{R}^{r \times r}$ is a free parameter matrix.
Since the matrix measure is convex \cite{desoer1972measure}, a sufficient certificate for $V_{R,p}$ to be a Lyapunov function can be obtained by solving the convex optimization problem
\begin{equation*}
    \min_{Z \in \R^{r \times r}} \mu_p\left(RAR^\dag + Z(I-RR^\dag)\right).
\end{equation*}
However, unlike the cases $p \in \{1, 2, \infty\}$, the matrix measure $\mu_p(\cdot)$ does not admit a closed-form expression for a general $p \in (1, \infty)$. 
Its evaluation requires solving a maximization problem over the $p$-norm unit sphere, which is generally not tractable in closed form.
This intractability motivates the two-step approach developed in the following subsections: first, certify $R$ and $\Gamma$ using the tractable $\mu_\infty$ condition; then, infer stability for all sufficiently large $p$.
This approach bypasses the direct computation of $\mu_p$.
 
\subsection{Step 1: Certifying a Polytopic Lyapunov Function}

Assume that $R$ has been chosen so that the sublevel sets of $\Psi_{R,\infty}(x) = \|Rx\|_\infty$ 
characterize the desired target region. It remains to verify that $\Psi_{R,\infty}$ is a Lyapunov 
function for the average system. 
Unlike $\mu_p$ for general $p\in(1,\infty)$, the matrix measure $\mu_\infty$ admits a closed-form expression~\cite{desoer1972measure}. Therefore, a sufficient certificate for $\Psi_{R,\infty}$ to be a Lyapunov function can be obtained by solving the convex optimization problem
\begin{equation}
    \min_{Z \in \R^{r \times r}} \mu_\infty\left(RAR^\dag + Z(I-RR^\dag)\right).
\end{equation}
Moreover, as shown in \cite[Lemma 1]{cinto2024polytopic}, this condition is not only sufficient but also necessary for $\Psi_{R,\infty}$ to be a Lyapunov function for the average system \eqref{eq:average_sys}.
For completeness, we restate this lemma below. Since its proof was omitted in the original reference, it is provided in Appendix~\ref{app:MRRA}.

\medskip
\begin{lemma}
    \label{lem:MRRA}
    Let $A\in\R^{n\times n}$ and let $R\in\R^{r\times n}$ be full column rank. Define $\Psi_{R,\infty}(x)\coloneqq\|Rx\|_\infty$.
    The following statements are equivalent:
    \begin{enumerate}[a)]
        \item there exists $\beta>0$ such that
            \begin{equation}
            D^+\Psi_{R,\infty}(x)\le -\beta \Psi_{R,\infty}(x)
            \qquad \forall x\in\R^n,
            \end{equation}
            where $D^+\Psi_{R,\infty}(x)$ denotes the upper-right Dini derivative of
            $\Psi_{R,\infty}(x)$ along the solutions of $\dot x=Ax$; \label{lem:1:lyap}
        \item there exists $\Gamma\in \R^{r\times r}$ such that $\mu_\infty(\Gamma) < 0$ and \label{lem:1:MRRA}
        \begin{equation} 
            \Gamma R = R A.
        \end{equation}  
    \end{enumerate}
\end{lemma}
\medskip

Given matrices $R$ and $\Gamma$ satisfying $\Gamma R = RA_{\lambda^*}$ and $\mu_\infty(\Gamma)<0$,  it remains to determine for which values of $p\in(1,\infty)$ the same matrix $R$ defines a valid $p$-norm Lyapunov function $V_{R,p}$ for the average system. 

\subsection{Step 2: Smoothing the Polytopic Lyapunov Function}

Rather than evaluating $\mu_p(\Gamma)$ directly, we bound it by means of the following interpolation property of matrix measures.

\medskip
\begin{lemma}
\label{lem:p_interpolation}
Let $\Gamma \in \R^{r \times r}$ and let $p_1, p_2 \in [1, \infty) \cup \{\infty\}$ satisfy $p_1 < p_2$. 
Then, for every $p \in (p_1, p_2)$, 
\begin{equation}
    \mu_p(\Gamma) \leq   \frac{\frac{1}{p} - \frac{1}{p_2}}{\frac{1}{p_1} - \frac{1}{p_2}} \mu_{p_1}(\Gamma) +\frac{\frac{1}{p_1} - \frac{1}{p}}{\frac{1}{p_1} - \frac{1}{p_2}}  \mu_{p_2}(\Gamma). 
\end{equation}
\end{lemma}
\medskip

\begin{proof}
By \cite[Theorem 7]{perov2017spectral}, the function $\alpha \mapsto \mu_{1/\alpha}(\Gamma)$ is convex on $\alpha \in [0, 1]$.
Equivalently, $\mu_p$ is convex as a function of $1/p$ on $[0,1]$. Since $p\in(p_1,p_2)$, we have $0\leq 1/p_2 < 1/p < 1/p_1 \leq 1$.
The claimed inequality then follows directly from convexity.
\end{proof}
\medskip

Lemma~\ref{lem:p_interpolation} implies that, if $\Gamma$ satisfies $\mu_2(\Gamma)<0$ and $\mu_\infty(\Gamma)<0$, then $\mu_p(\Gamma) < 0$ for all $p\in (2,\infty)$, without requiring any additional computation.
More generally, the lemma allows $\mu_p(\Gamma)$ to be bounded for any $p$ using only the computable extreme cases $\mu_1$ and $\mu_\infty$, thereby enabling $p$ to be treated as a continuous tuning parameter in the switching logic \eqref{eq:maps:g}.
Building on Lemma~\ref{lem:p_interpolation}, we next show that the $\infty$-norm certificate $\mu_\infty(\Gamma) < 0$ implies $\mu_p(\Gamma) < 0$ for all $p$ above a computable threshold.

\medskip
\begin{theorem}
\label{thm:underp}
Suppose $\Gamma \in \R^{r \times r}$ satisfies $\mu_\infty(\Gamma) < 0$.
Then the following statements hold:
\begin{enumerate}[a)]
\item If $\mu_1(\Gamma) < 0$, then $\mu_p(\Gamma) < 0$ for all $p \in (1,\infty)$. \label{thm:up:1}

\item If $\mu_1(\Gamma) \geq 0$, then $\mu_p(\Gamma) < 0$ for all $p\in (\underline{p} ,\infty)$, where \label{thm:up:2}
\begin{equation}
    \underline{p} \coloneqq 1 + \max_{i\in\N_r} \frac{[\Metz(\Gamma)^\top \one]_i}{-[\Metz(\Gamma) \one]_i} \geq 1.
    \label{eq:underlinep}
\end{equation} 
\end{enumerate}
\end{theorem}
\medskip

\begin{proof}
Item \ref{thm:up:1}) follows directly from Lemma~\ref{lem:p_interpolation}.

To prove item \ref{thm:up:2}), assume $\mu_1(\Gamma) \geq 0$ and let $p\in(\underline{p},\infty)$.
For brevity, define $\rho_i\coloneqq [\Metz(\Gamma)\one]_i$, $C_i \coloneqq [\Metz(\Gamma)^\top\one]_i$, which are the row and column sums, respectively.
From the definition of $\mu_\infty$ \cite{desoer1972measure}, we have $\mu_\infty(\Gamma) = \max_{i\in\N_r} \rho_i$.
Since $\mu_\infty(\Gamma) < 0$ by hypothesis, it follows that $\rho_i < 0$ for all $i \in \N_r$.
Moreover, $\mu_1(\Gamma) = \max_{i\in\N_r} C_i \geq 0$, so there exists at least one index $i \in \N_r$ such that $C_i \geq 0$. Consequently $- C_i / \rho_i \geq 0$, and therefore $\underline{p} \geq 1$.

By \cite[Theorem~8]{perov2017spectral}, for every $p>\underline{p}\geq 1$, 
\begin{align} 
&\mu_p(\Gamma) \leq \max_{i\in\N_r} \Big(\Gamma_{i,i} + \frac{p-1}{p} \sum_{\substack{j\in\N_r\\j\neq i}} \abs{\Gamma_{i,j}} + \frac{1}{p} \sum_{\substack{j\in\N_r\\j\neq i}} \abs{\Gamma_{j,i}} \Big) \notag\\
&= \max_{i\in\N_r} \frac{(p-1)\rho_i + C_i}{p}  < \max_{i\in\N_r} \frac{(\underline{p}-1)\rho_i + C_i}{p} \label{eq:mu_p2} \\
&\leq \frac{1}{p} \max_{i\in\N_r} \left(- C_i + C_i \right) = 0. \label{eq:mu_p3}
\end{align}
The strict inequality in \eqref{eq:mu_p2} follows from $p > \underline{p}$ and $\rho_i < 0$ for all $i \in \N_r$. 
Inequality \eqref{eq:mu_p3} follows from the definition of $\underline{p}$, which implies $(\underline{p}-1) \rho_i \leq -C_i $ for all $i \in \N_r$. 
Hence, $\mu_p(\Gamma) < 0$ for all $p \in (\underline{p}, \infty)$, concluding the proof.
\end{proof}
\medskip

\begin{remark}
The threshold $\underline{p}$ defined in \eqref{eq:underlinep} quantifies the asymmetry of the Metzler matrix $\Metz(\Gamma)$. 
If $\Metz(\Gamma)$ is symmetric, then its row and column sums coincide, and therefore $\underline{p} = 1$. In that case, $\mu_p<0$ for all $p \in (1,\infty)$. 
As the off-diagonal entries become more asymmetric, larger values of $p$ are required to compensate for this asymmetry.
\end{remark}
\medskip

Theorem~\ref{thm:underp} completes the method: given $R$ and $\Gamma$ from Step~1 with $\mu_\infty(\Gamma)<0$, any $p\in(\underline{p},\infty)$ yields a valid $p$-norm Lyapunov function $V_{R,p}$ that can be used directly in the hybrid control strategies of Sections~\ref{sec:hybrid}--\ref{sec:dwell_time}, with convergence regions that approximate the prescribed polytopic target region.

In this way, the contribution of the paper is not merely the use of weighted $p$-norm Lyapunov functions, but the combination of a smooth $p$-norm extension of the steepest-decrease strategy with a constructive route from polytopic certificates to Lyapunov functions that are directly usable under unconstrained and dwell-time switching.


\section{Examples}
\label{sec:examples}

Consider the switched affine system \eqref{eq:sas} with matrices 
\begin{equation}
    A_1\coloneqq\begin{bsmallmatrix}
        -1 & \phantom{-}2 \\ -2 & -1
    \end{bsmallmatrix}, \medspace \medspace
    A_2\coloneqq\begin{bsmallmatrix}
        -1 & -2 \\ \phantom{-}2 & -1
    \end{bsmallmatrix},  \medspace \medspace
    b_1=-b_2\coloneqq\begin{bsmallmatrix}
        2.5 \\ 0 
    \end{bsmallmatrix},
    \label{eq:example1:A}
\end{equation}
taken from \cite[Example 4.1]{cinto2024polytopic}.
Choosing $\lambda^* \coloneqq [0.5,\,0.5]^\top$ yields $A_{\lambda^*} = -I_2$ and $b_{\lambda^*}=0$, so Assumption~\ref{ass:1} holds.

Let $R = I_2$ and suppose the desired target region is given by a sublevel set of $\Psi_{R,\infty}(x) = \norm{x}_\infty$. 
Defining $\Gamma = A_{\lambda^*} = -I_2$, we obtain $\Gamma R = R A_{\lambda^*}$ and $\mu_\infty(\Gamma)=-1<0$.
Therefore, by Lemma~\ref{lem:MRRA}, $\Psi_{R,\infty}$ is a Lyapunov function for the average system.
As shown in \cite{cinto2024polytopic}, however, this polytopic Lyapunov function fails to stabilize the system 
defined by \eqref{eq:example1:A}, even when the time regularization technique is applied.
On the other hand, since $\mu_1(\Gamma)=-1<0$, Theorem~\ref{thm:underp} guarantees $\mu_p(\Gamma)<0$ for all $p \in (1, \infty)$.
Hence, by Lemma~\ref{lem:Gamma}, $V_{R,p}$ is a Lyapunov function for the average system satisfying \eqref{eq:p-norm-lyap-cross} with $S \coloneqq (p\,|\mu_p(\Gamma)|)^{1/p} R = p^\frac{1}{p} I_2$ for all $p \in (1, \infty)$.

Under the hybrid strategy of Section~\ref{sec:hybrid}, the set $\setA=\{0\}\times\N_N$ is UGAS for any $p \in (1,\infty)$, and the choice of $p$ has limited practical impact in this setting, since switching becomes arbitrarily fast as the continuous state approaches the origin regardless of $p$. The main benefit of choosing $p \neq 2$ arises under dwell-time constraints, where the convergence region is a sublevel set of $V_{R,p}$ whose geometry depends on $p$. We illustrate this effect by applying both regularization techniques of Section~\ref{sec:dwell_time}.

\subsection{Space Regularization}

We apply the space regularization of Section~\ref{subsec:space_regularization} with $\eta = 0.5$ and investigate the effect of varying $p$ over the range $[2, 20]$.
For $p = 2$, we set $\varepsilon_2 \coloneqq 0.05$, so that the target region is $\Omega_2 = \{x : V_{R,2}(x) \leq \varepsilon_2\}$. 
For each $p > 2$, define
\begin{equation}
    \varepsilon_p \coloneqq \frac{1}{p} \left( 
    \frac{\sqrt{2\varepsilon_2}}{2^{1/2 - 1/p}} 
    \right)^p, \label{eq:epsilon_p}
\end{equation}
and  
$\Omega_p \coloneqq \{x : V_{R,p}(x) \leq \varepsilon_p\}$. This choice ensures that 
$\Omega_{p_2} \subseteq \Omega_{p_1}$ for all $p_2 \geq p_1 \geq 2$.
Therefore, any improvement in dwell time as $p$ increases cannot be attributed to an enlargement of the target region.

For selected values of $p$ in $[2, 20]$, we simulate the space-regularized system $\Hy_\varepsilon$ from $500$ initial conditions uniformly distributed along a circle of radius $1.5$ in the state space, up to a final simulation time of $t_{\mathrm{final}} = 5$. 
For each trajectory, dwell times are classified as exterior if they occur before the first entry into $\Omega_p$, and as interior thereafter.

\begin{figure}[htb]
\centering
\includegraphics[
    width=0.98\linewidth
]{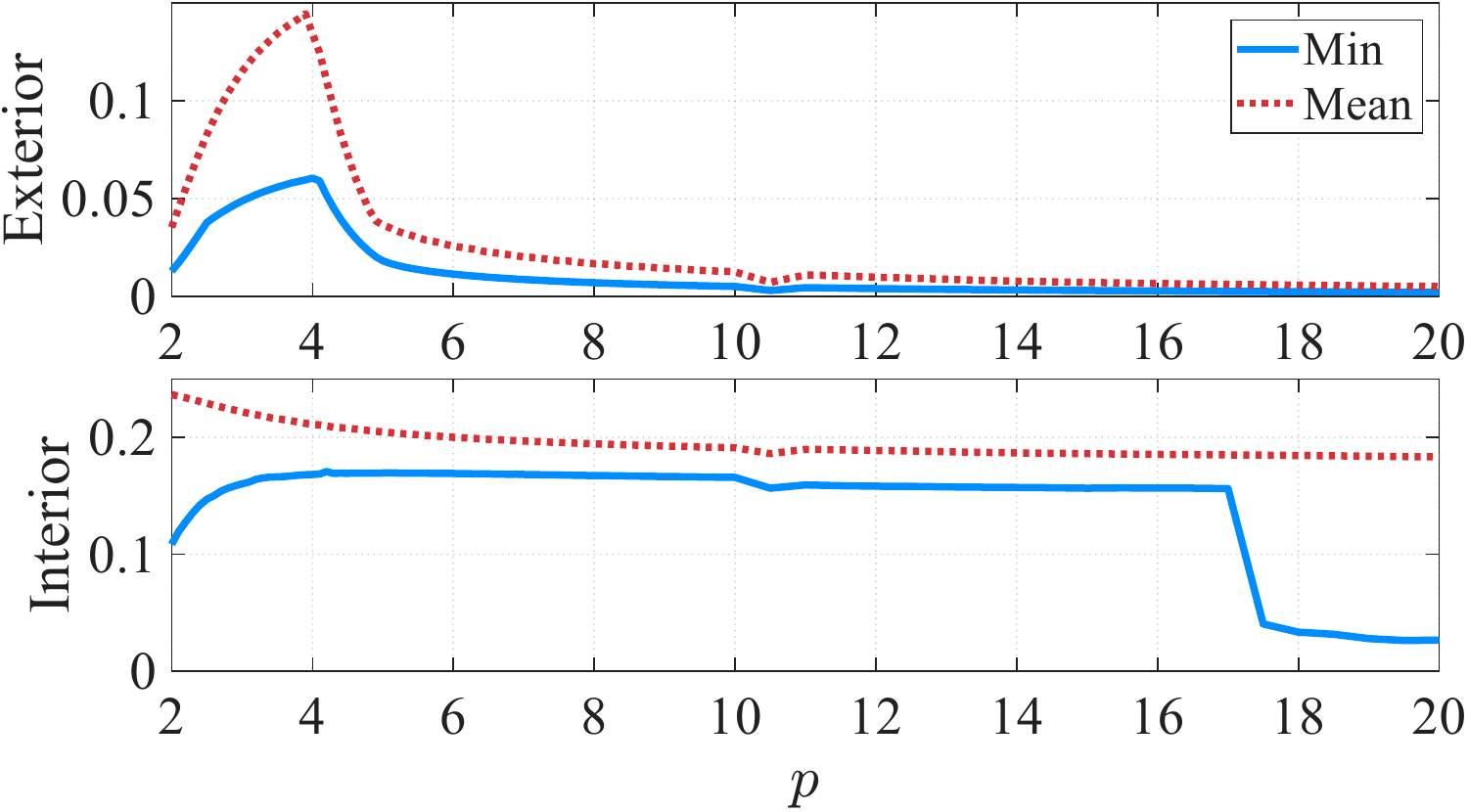}
\caption{Minimum and mean exterior and interior dwell-times over $500$ initial conditions, as functions of $p$, under space regularization with $\eta = 0.5$ and $\varepsilon_p$ as in \eqref{eq:epsilon_p}.}
\label{fig:space_reg_times}
\end{figure}

\begin{figure*}[htb]
\centering
    \begin{subfigure}[b]{0.325\textwidth}
        \centering
        \includegraphics[width=\textwidth]{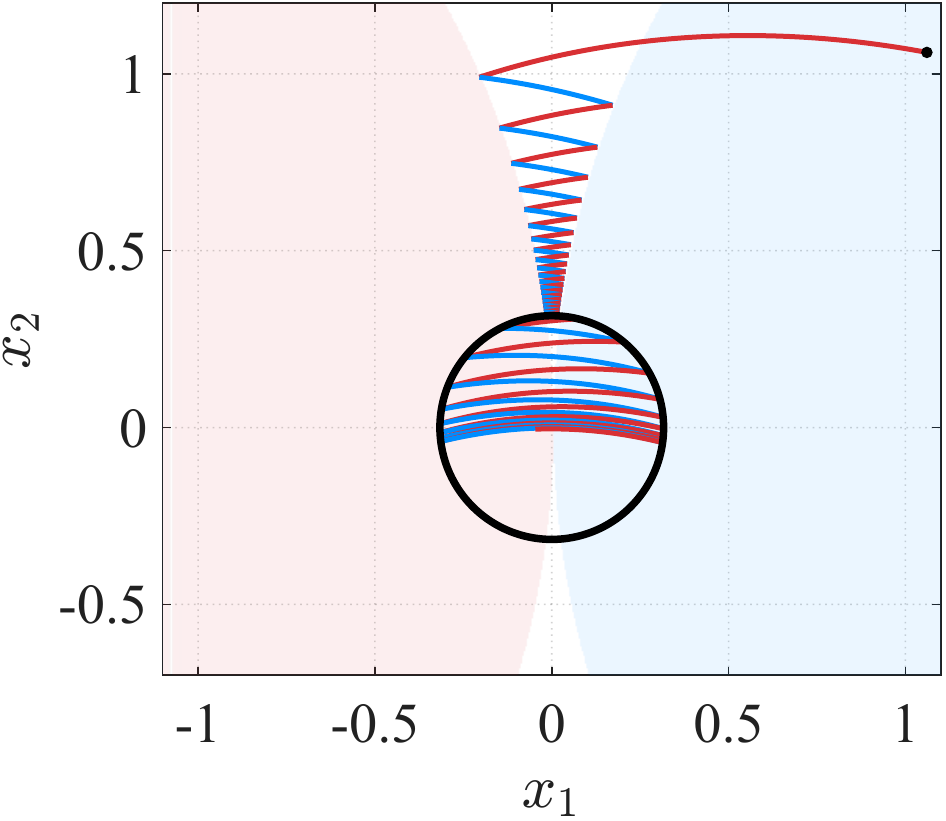}
        \caption{$p=2$}
        \label{fig:p2}
    \end{subfigure}
    \hfill
    \begin{subfigure}[b]{0.325\textwidth}
        \centering
        \includegraphics[width=\textwidth]{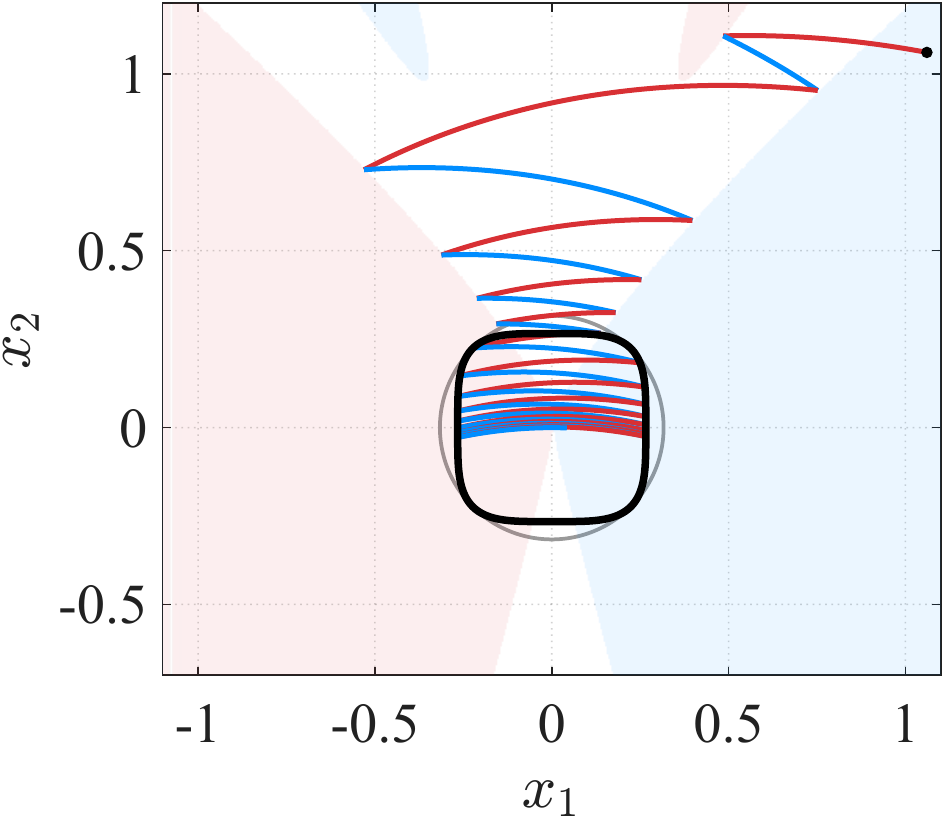}
        \caption{$p=4$}
        \label{fig:p4}
    \end{subfigure}
    \hfill
    \begin{subfigure}[b]{0.325\textwidth}
        \centering
        \includegraphics[width=\textwidth]{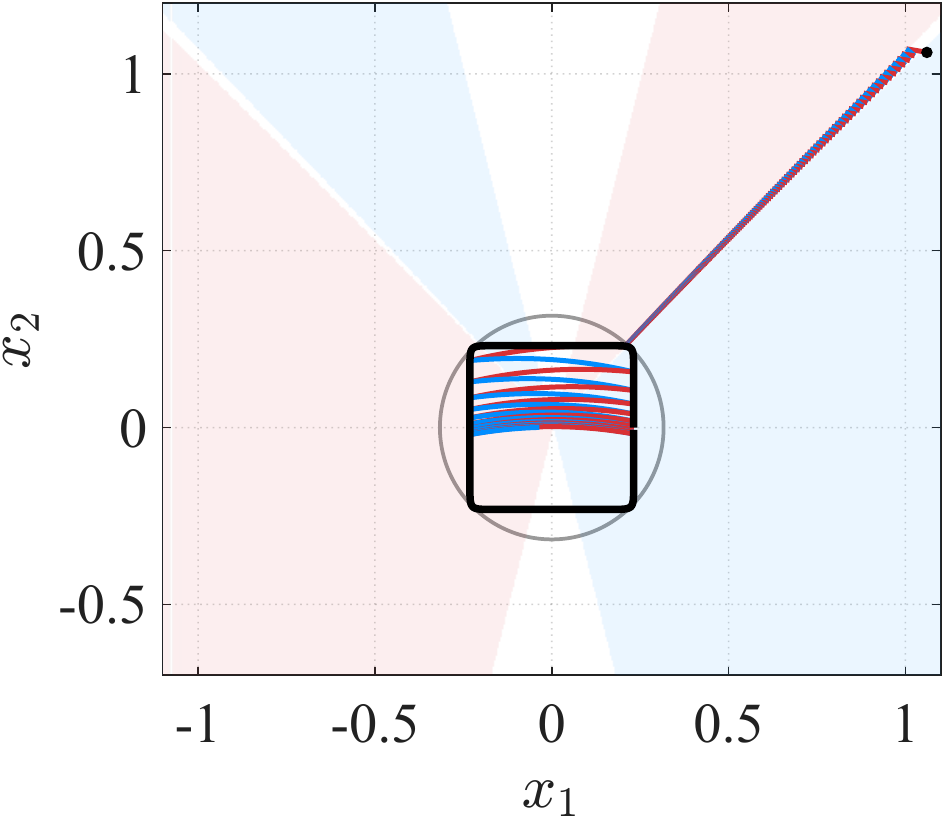}
        \caption{$p=20$}
        \label{fig:p20}
    \end{subfigure}
\caption{Closed-loop trajectories under space regularization with $\eta=0.5$, for $p=2$, $4$, and $20$. 
The shaded regions indicate where the condition $\inner{\nabla V_{R,p}(x)}{f_k(x)}\geq-\eta V_{S,p}(x)$ holds for subsystem~1 (blue) and 2 (red).
In each plot, a trajectory starting from $x_0 = 1.5[\cos(45^\circ),\, \sin(45^\circ)]^\top$ (black dot) is shown, with line color indicating the active mode. The boundary of the target region $\Omega_p$ is shown in black, while the boundary of $\Omega_2$ (gray circle) highlights the inclusion $\Omega_p \subset \Omega_2$.}
\label{fig:space_reg_comp}
\end{figure*}

Figure~\ref{fig:space_reg_times} shows the minimum and mean dwell times, for both the exterior and interior phases, as functions of $p$. Compared with the quadratic case $p = 2$, the minimum exterior dwell-time increases with $p$, reaches a factor of $4.67$ at $p = 4$ (while the mean increases by a factor of $3.79$), and then decreases for larger values of $p$.
This decrease is consistent with the fact that, as $p \to \infty$, the $p$-norm Lyapunov function approaches the polytopic function $\Psi_{R,\infty}$, which, as shown in \cite{cinto2024polytopic}, fails to stabilize this system under the steepest-decrease strategy.
The minimum interior dwell-time also improves up to $p \approx 17$, reaching a factor of $1.55$ at $p = 4$ relative to the quadratic case, despite the fact that $\Omega_p$ becomes smaller as $p$ increases. 
Beyond that range, the interior dwell-time decreases, as the increasingly sharp corners 
of $\Omega_p$ allow trajectories to leave the target region more easily.
Figure~\ref{fig:space_reg_comp} illustrates this effect on individual trajectories by comparing the cases $p = 2$, $4$, and $20$, all starting from the initial condition $x_0 = 1.5[\cos(45^\circ),\, \sin(45^\circ)]^\top$. 
The trajectory for $p = 4$ exhibits fewer switches in both exterior and interior phases, while  $\Omega_{20}\subset\Omega_4\subset\Omega_2$.

\subsection{Time Regularization}

We now illustrate the time regularization of Section~\ref{subsec:time_regularization} for $p = 4$ and $\eta = 0.5$. Two values of the minimum dwell-time parameter are considered: $T_1 = 0.06$ and $T_2 = 0.03$. The value $T_1$ is chosen close to the minimum exterior dwell-time observed for $p = 4$ under space regularization (see Figure~\ref{fig:space_reg_times}).

Figure~\ref{fig:time_reg_comp} shows the closed-loop trajectories from two representative initial conditions. For $T_1$, the switching frequency is significantly lower than for $T_2$, resulting in longer intervals of continuous evolution between jumps. As expected from Theorem~\ref{thm:time_reg}, increasing $T$ reduces the switching frequency, at the cost of convergence to a larger neighborhood of the origin.

\begin{figure}[ht]
\centering
    \begin{subfigure}[b]{0.48\linewidth}
        \centering
        \includegraphics[width=\textwidth,
                        trim=0.7cm 0.2cm 2.2cm 1.2cm,
                        clip]{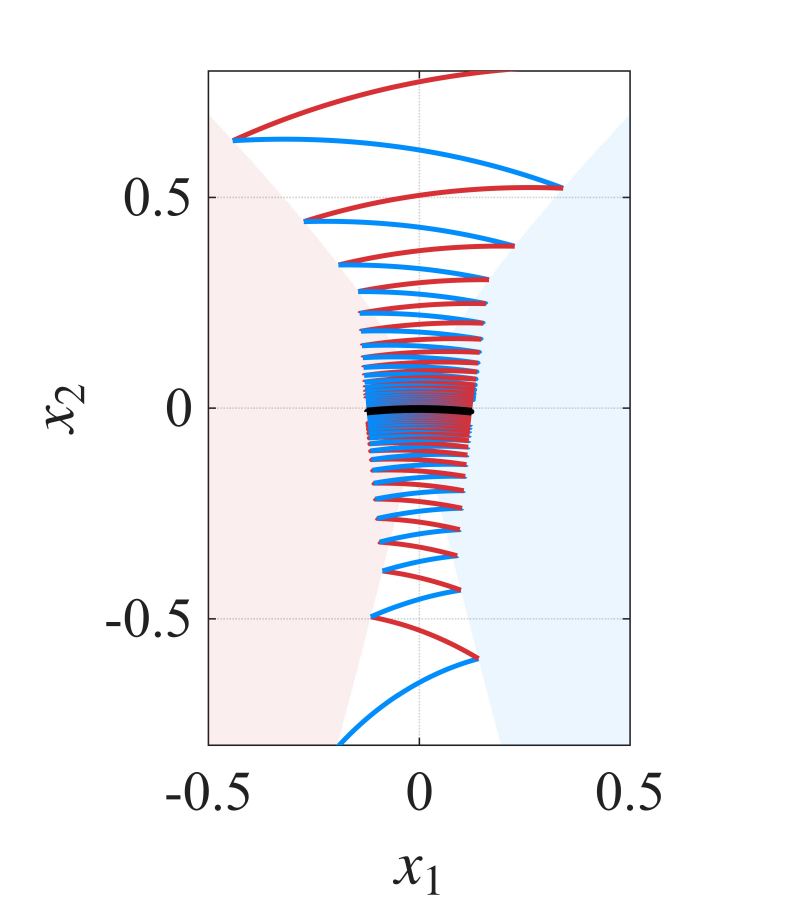}
        \caption{$T_1=0.06$}
        \label{fig:T1}
    \end{subfigure}
    \hfill
    \begin{subfigure}[b]{0.48\linewidth}
        \centering
        \includegraphics[width=\textwidth,
                        trim=0.7cm 0.2cm 2.2cm 1.2cm,
                        clip]{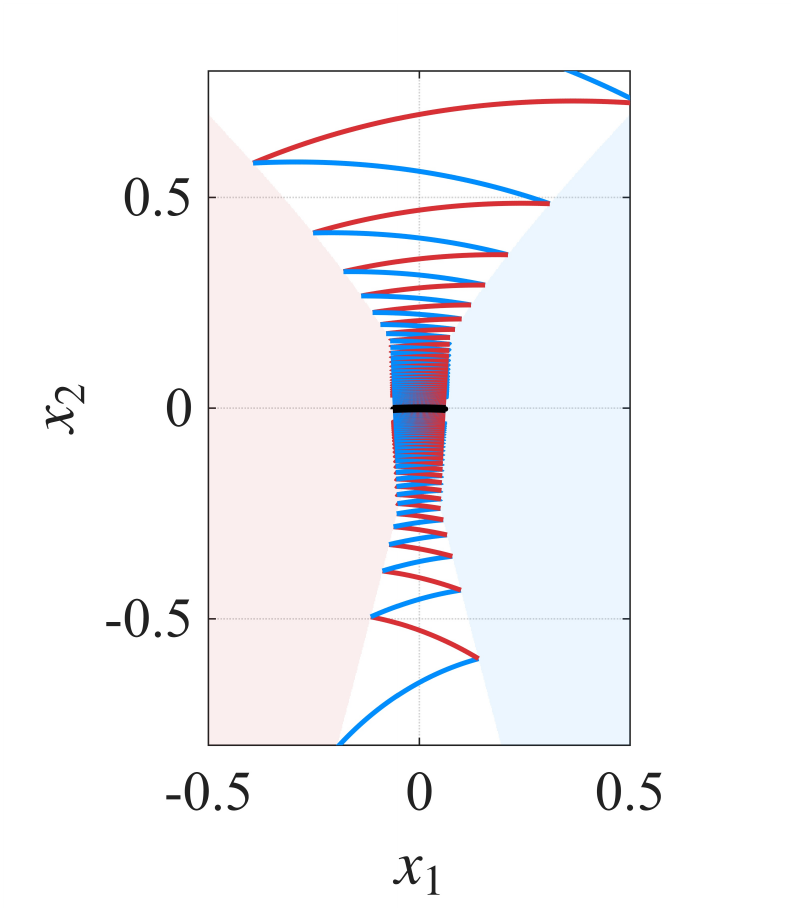}
        \caption{$T_2=0.03$}
        \label{fig:T2}
    \end{subfigure}
\caption{Closed-loop trajectories under time regularization with $\eta = 0.5$ and $p = 4$, for $T_1 = 0.06$ (left) and $T_2 = 0.03$ (right). 
The shaded regions indicate where the condition $\inner{\nabla V_{R,p}(x)}{f_k(x)}\geq-\eta V_{S,p}(x)$ holds for subsystem~1 (blue) and subsystem 2 (red).
Trajectories start from $x_0^{(1)} = [0,\, 1.5]^\top$ and $x_0^{(2)} = [0,\, -1.5]^\top$, and are simulated up to $10$ time units; the portion corresponding to $t \ge 6$ is shown in black to highlight the switching behavior near the convergence region.}
\label{fig:time_reg_comp}
\end{figure}

\section{Conclusion}\label{sec:conclusion}

In this paper, we extended the steepest-decrease switching strategy for switched affine systems from quadratic to weighted $p$-norm Lyapunov functions, motivated by the fact that the direct extension of the strategy may fail in the $\infty$-norm (i.e., polytopic) case. 
Assuming the existence of a Hurwitz convex combination with zero affine term, we showed that the proposed strategy guarantees UGAS under unconstrained switching and practical UGAS under dwell-time constraints via two regularization approaches.

On the theoretical side, we established a characterization of the Hurwitz property in terms of weighted $p$-norm Lyapunov functions, thereby generalizing the classical quadratic Lyapunov equivalence to the family $p \in (1,\infty)$. 
On the practical side, we developed a constructive two-step method that starts from a polytopic Lyapunov certificate for the average system and yields a smooth weighted $p$-norm Lyapunov function for all sufficiently large $p$. 
In particular, once the certificate $\mu_\infty(\Gamma)<0$ is established, the same matrix $R$ yields a valid weighted $p$-norm Lyapunov function suitable for the proposed hybrid-control and dwell-time developments. 

Taken together, these results show that the paper's contribution does not lie merely in replacing quadratic Lyapunov functions by another smooth family, but in combining a weighted $p$-norm extension of the steepest-decrease strategy with a constructive route from polytopic certificates to Lyapunov functions that are directly usable under both unconstrained and dwell-time switching. This construction avoids the nonsmoothness issues of direct polytopic designs while preserving, in smooth form, the geometry of the desired target region. Numerical examples illustrate that choosing $p>2$ can improve dwell-time performance.

Future work includes designing Lyapunov functions when the polytopic function associated with the desired target region is not itself a Lyapunov function for the average system.


\appendices
\renewcommand{\thesection}{\Alph{section}}
\renewcommand{\thesectiondis}{\Alph{section}}

\renewcommand{\thesubsection}{\thesection.\arabic{subsection}}
\renewcommand{\thesubsectiondis}{\thesection.\arabic{subsection}}

\numberwithin{equation}{section}

\renewcommand{\thetheorem}{\thesection.\arabic{theorem}}

\section{Construction of p-norm Lyapunov Functions}
\label{app:Lyap}
This appendix presents the results required to construct a global $p$-norm Lyapunov function using the system's Jordan decomposition. Lemma~\ref{lemma:Lyap_Aggregation} provides a constructive procedure for a global $p$-norm Lyapunov function by combining the local stability properties of the individual subsystems represented by the diagonal blocks of a Jordan matrix $J$. Lemma~\ref{lem:Jordan_Complex_Block} covers the case of a real Jordan block associated with complex eigenvalues, while Lemma~\ref{lem:Jordan_Real_Block} addresses the case of real eigenvalues. By applying Lemmas~\ref{lem:Jordan_Complex_Block} and \ref{lem:Jordan_Real_Block} to each individual Jordan block and subsequently utilizing Lemma~\ref{lemma:Lyap_Aggregation}, one can construct a global $p$-norm Lyapunov function for any linear system in Jordan canonical form.

\medskip
\begin{lemma}
\label{lemma:Lyap_Aggregation}
Let $J \in \R^{n\times n}$ be a block-diagonal matrix of the form
\begin{equation}
J\coloneqq \text{diag}(B_1, B_2, \dots, B_m)
\end{equation}
where $B_i \in \R^{n_i \times n_i}$ for all $i \in \N_m$ satisfying $\sum_{i=1}^m n_i = n$.

Let $p\in(1,\infty)$ and assume that for each $i \in \N_m$ there exist a full column rank matrix $R_i \in \R^{r_i \times n_i}$ and a scalar $\beta_i > 0$ such that the function
$V_{R_i, p}(z_i)$ satisfies 
\begin{equation}
    \inner{\nabla V_{R_i,p}(z_i)}{B_i z_i} \leq -\beta_i V_{R_i, p}(z_i)
    \label{eq:block_decay}
\end{equation}
for all $z_i \in \R^{n_i}$.
Define the block-diagonal matrix 
\begin{equation}
R\coloneqq \text{diag}(R_1, R_2, \dots, R_m) \in \R^{(\sum_{i=1}^m r_i) \times n}.
\end{equation}

Then, the function
$V_{R, p}(z)$ satisfies
\begin{equation}
    \inner{\nabla {V}_{R, p}(z)}{Jz} \leq -\beta V_{R, p}(z),
    \label{eq:global_decay}
\end{equation}
for all $z\in\R^n$,
with $\beta \coloneqq  \min_{i\in\N_m} \beta_i$.
\end{lemma}
\medskip

\begin{proof}
    Given the block-diagonal structure of $J$, we can partition $z$ into $m$ vectors consistent with the dimensions of the blocks $B_{n_i}$ as
    $z= \begin{bmatrix} z_1^\top & \dots & z_m^\top \end{bmatrix}^\top$,
    where $z_i \in \R^{n_i}$. Consequently, 
    \begin{equation}
        Jz = \begin{bmatrix}
            (B_1 z_1)^\top & \dots & (B_m z_m)^\top 
        \end{bmatrix}^\top. \label{eq:Jz}
    \end{equation}
    
     Similarly, due to the block-diagonal structure of $R$, the vector $Rz$ can be expressed as the concatenation
       $Rz = \begin{bmatrix} (R_1 z_1)^\top \medspace \dots \medspace (R_m z_m)^\top \end{bmatrix}^\top$.
    Then, considering the definition of $V_{R, p}$, we have
    \begin{equation}
        V_{R, p}(z) = \frac{1}{p} \sum_{i=1}^m \norm{R_i z_i}_p^p = \sum_{i=1}^m V_{R_i, p}(z_i). \label{eq:sum_lyap}
    \end{equation}
    Since each $V_{R_i, p}(z_i)$ depends exclusively on $z_i$,
    \begin{equation}
        \nabla {V}_{R, p}(z) = \begin{bmatrix} \nabla {V}_{R_1, p}(z_1) & \dots & \nabla {V}_{R_m, p}(z_m) \end{bmatrix}. \label{eq:nablaR}
    \end{equation}
    Multiplying \eqref{eq:nablaR} by \eqref{eq:Jz}, and applying \eqref{eq:block_decay}, we obtain
    \begin{align*}
    &\inner{\nabla V_{R, p}(z)}{Jz} = \sum_{i=1}^m \inner{\nabla V_{R_i, p}(z_i)}{B_i z_i} \notag\\
    &\quad\leq \sum_{i=1}^m -\beta_i V_{R_i, p}(z_i) \leq -\beta \sum_{i=1}^m V_{R_i, p}(z_i) 
        = -\beta V_{R, p}(z),
    \end{align*}
    which completes the proof.
\end{proof}

\medskip

\begin{lemma}
\label{lem:Jordan_Complex_Block}
    Let $C_{n}(\alpha, \omega) \in \R^{2n \times 2n}$ be a matrix of the form
    \begin{equation}
        C_{n}(\alpha, \omega) \coloneqq 
        I_n \otimes C(\alpha, \omega) + \mathcal{N}_n \otimes I_2,
        \label{eq:block_complex}
    \end{equation}
    where $C(\alpha, \omega) \coloneqq   \begin{bsmallmatrix} -\alpha & \omega \\ -\omega & -\alpha \end{bsmallmatrix}$ with $\alpha$, $\omega > 0$.
    Choose $m \in \N$ such that
    \begin{equation}
    \label{eq:m}
    \zeta \coloneqq  \alpha - \omega \tan \left(\frac{\pi}{2m} \right) > 0
    \end{equation}
and let
    $Q(m) \in \R^{m \times 2}$ be a matrix with rows
\begin{equation}
Q(m)_{i:}\coloneqq [\cos((i-1)\pi/m),\,  \sin((i-1)\pi/m)]
\end{equation}
for all $i\in\N_{m}$. 
For any $\varepsilon \in (0, \zeta)$, define the diagonal matrix
\begin{equation}
\label{eq:E}
E(\varepsilon,n)\coloneqq \text{diag}(\varepsilon^{n-1}, \varepsilon^{n-2}, \dots, 1),
\end{equation}
and 
the block-diagonal matrix $R(\varepsilon,n,m) \in \R^{(mn) \times 2n}$ as
\begin{equation}
    R(\varepsilon,n,m) \coloneqq  E(\varepsilon,n) \otimes Q_m.
\end{equation}
    Then, for any given $p \in (1, \infty)$, the function $V_{R(\varepsilon,n,m), p}(z)$ satisfies
    \begin{equation}
        \inner{ \nabla V_{R(\varepsilon,n,m), p}(z)}{C_{n}(\alpha, \omega) z} \leq -\beta V_{R(\varepsilon,n,m), p}(z),
    \end{equation}
    with $\beta \coloneqq  p(\zeta - \varepsilon) > 0$.
\end{lemma}
\medskip

\begin{proof}  
    Partition the state $z$ into $n$ substates $z_1, \dots, z_n$ with $z_i\in\R^2$ for all $i\in\N_n$.
    Then, we have
    \begin{align}
        \dot{z}_i &= C(\alpha, \omega) z_i + z_{i+1} \quad \text{for } i\in\N_{n}, \label{eq:dotx_i}
    \end{align}
    where $z_{n+1}\coloneqq  0$.
    Define $y\coloneqq  R(\varepsilon,n,m)z \in \R^{mn}$ and partition $y$ into $n$ substates $y_1, \dots, y_n$ with $y_i\in\R^m$ for all $i\in\N_n$.
    By the block diagonal form of $R(\varepsilon,n,m)$, it results
    \begin{align}
        y_i &= \varepsilon^{n-i} Q(m) z_i \quad \forall i\in\N_n. \label{eq:z_i}
    \end{align}
    Differentiating \eqref{eq:z_i} and using \eqref{eq:dotx_i}, we have
    \begin{align}
        \label{eq:y_dotQ}
        \dot y_i &= \varepsilon^{n-i} Q(m) \dot z_i = \varepsilon^{n-i} Q(m) (C(\alpha, \omega) z_i + z_{i+1})
    \end{align}
 for all $i\in\N_n$. 
    From \cite{polanski1995infinity}, we have 
    \begin{equation}
        \label{eq:QCWQ}
        Q(m) C(\alpha, \omega) = W(\alpha, \omega) Q(m),    
    \end{equation}
    where $W(\alpha, \omega) \in \R^{m\times m}$ is defined as the skew-circulant matrix
    \begin{equation}
        W(\alpha, \omega)\coloneqq  \scalebox{0.8}{$\begin{bmatrix}
            c & d & \cdots & 0\\
            0 & c & \ddots & 0\\
            \vdots & \ddots & \ddots & d\\
            -d & \cdots & 0 & c
        \end{bmatrix}$},
    \end{equation}
    with $c\coloneqq  -\zeta - d$ and $d\coloneqq  \frac{\omega}{\sin(\frac{\pi}{m})}$.
    Replacing \eqref{eq:QCWQ} in \eqref{eq:y_dotQ} it results
    \begin{align*}
        \dot y_i &= \varepsilon^{n-i} (W(\alpha, \omega) Q(m) z_i + Q(m) z_{i+1}) \nonumber\\
         &= W(\alpha, \omega) y_i + \varepsilon y_{i+1}  \qquad\qquad \forall i\in\N_{n},
    \end{align*}
    where $y_{n+1}\coloneqq 0$.
    Then, by properties of the matrix measure $\mu_p$ \cite{soderlind2006logarithmic}, it follows that
    \begin{align}
        \frac{d}{dt} \norm{y_i}_p &\leq \mu_p(W(\alpha, \omega)) \norm{y_i}_p + \norm{\varepsilon y_{i+1}}_p  \; \forall i\in\N_{n}.
    \end{align}
    Define $\dot V(z) \coloneqq \inner{ \nabla V_{R(\varepsilon,n,m), p}(z)}{C_{n}(\alpha, \omega) z}$. Thus, we have
    \begin{align}
        &\dot{V}(z) = \frac{d}{dt} \left[ \frac{1}{p} \sum_{i=1}^n \norm{y_i}_p^{p} \right] = \sum_{i=1}^n \norm{y_i}_p^{p-1} \frac{d}{dt} \norm{y_i}_p \nonumber\\
        &\leq \sum_{i=1}^n \norm{y_i}_p^{p-1} (\mu_p(W(\alpha, \omega)) \norm{y_i}_p + \norm{\varepsilon y_{i+1}}_p) \nonumber\\
        &= \mu_p(W(\alpha, \omega)) \norm{y}_p^p + \varepsilon \sum_{i=1}^{n-1} \norm{y_i}_p^{p-1} \norm{y_{i+1}}_p. \label{eq:V_R_complex}
    \end{align}  
    Using Young's Inequality, we have
    \begin{align}
        &\sum_{i=1}^{n-1} \norm{y_i}_p^{p-1} \norm{y_{i+1}}_p \leq \sum_{i=1}^{n-1} \frac{p-1}{p} \norm{y_i}_p^p + \frac{1}{p} \norm{z_{i+1}}_p^p \nonumber\\
        &= \tfrac{p-1}{p} \norm{y_1}_p^p + \sum_{i=2}^{n-1}\norm{y_i}_p^p + \frac{1}{p} \norm{y_n}_p^p \leq \sum_{i=1}^{n}\norm{y_i}_p^p, \label{eq:young}
    \end{align}
    where the last inequality follows from the fact that $p>1$.
    Using \eqref{eq:young} in \eqref{eq:V_R_complex}, it results
    \begin{align}
        \label{eq:V_R_pre}
        \dot{V}(z) &\leq \left(\mu_p(W(\alpha, \omega)) + \varepsilon \right) \norm{y}_p^p.
    \end{align}

    On the other hand, given that $m\geq 2$, it follows that $d>0$. Then,
     $\mu_\infty(W(\alpha, \omega)) = \mu_1(W(\alpha, \omega)) = c + \abs{d}=c+d = - \zeta.$
    By the previous equation
    and using Lemma~\ref{lem:p_interpolation} with $p_1\coloneqq 1$ and $p_2\coloneqq \infty$, for all $p \in (1, \infty)$ we have
    \begin{align}
        \label{eq:mu_p_zeta}
        \mu_p(W(\alpha, \omega)) &\leq \max(\mu_1(W(\alpha, \omega)), \mu_\infty(W(\alpha, \omega))) \nonumber\\
        &= -\zeta.
    \end{align}
    Using \eqref{eq:mu_p_zeta} in \eqref{eq:V_R_pre}, it follows that
    \begin{equation}
        \dot{V}(z) \leq -(\zeta - \varepsilon) \norm{y}_p^p  = - \beta V_{R(\varepsilon,n,m), p}(y),
    \end{equation}
    and the proof is complete.
\end{proof}

\medskip
\begin{lemma}
\label{lem:Jordan_Real_Block}
Let $J_n(s) \in \R^{n\times n}$ be a matrix of the form
 \begin{equation}
        J_n(s) \coloneqq  
        s I_n + \mathcal{N}_n
        \label{eq:block_real}
    \end{equation}
 with $s < 0$.    
For any $\varepsilon \in (0, |s|)$, define
$R(\varepsilon,n)\coloneqq E(\varepsilon,n)$ with $E(\varepsilon,n)$ as defined in \eqref{eq:E}.
Then, for any $p \in (1, \infty)$, the function
        $V_{R(\varepsilon,n), p}(z) \coloneqq  \frac{1}{p} \norm{R(\varepsilon,n) z}_p^p$
    satisfies
    \begin{equation}
        \inner{{V}_{R(\varepsilon,n), p}(z)}{J_n(s) z} \leq -\beta V_{R(\varepsilon,n), p}(z),
        \label{eq:global_decay_real}
    \end{equation}
    with $\beta \coloneqq  p(|s| - \varepsilon) > 0$.
\end{lemma}
\medskip

\begin{proof}
    Define $y \coloneqq  R(\varepsilon,n) z$. 
    Then, we have
        $\dot{y} = R(\varepsilon,n) \dot{z} = R(\varepsilon,n) J_n(s) R(\varepsilon,n)^{-1} y$
        with
    \begin{equation*}
         R(\varepsilon,n) J_n(s) R(\varepsilon,n)^{-1} = 
        s I_n + \varepsilon \mathcal{N}_n.
    \end{equation*}
    Using the properties of the matrix measure $\mu_p$ \cite{desoer1972measure}, it follows that
    \begin{align}
        \frac{d}{dt} \norm{y}_p &\leq \mu_p(s I_n + \varepsilon \mathcal{N}_n) \norm{y}_p  \leq \left(s + \varepsilon \mu_p(\mathcal{N}_n)\right) \norm{y}_p \nonumber\\
        & \leq \left(s + \varepsilon \norm{\mathcal{N}_n}_p\right) \norm{y}_p  = - \left(\abs{s} - \varepsilon\right) \norm{y}_p, \label{eq:derivativey}
    \end{align}
    where we have used the fact that $\norm{\mathcal{N}_n}_p=1$ for all $p \in (1,\infty)$.
    Define $\dot V(z)\coloneqq\inner{{V}_{R(\varepsilon,n), p}(z)}{J_n(s) z}$.
    Considering \eqref{eq:derivativey}, we have
    \begin{align}
        &\dot{V}(z) = \norm{y}_p^{p-1} \frac{d}{dt}\norm{y}_p \nonumber \leq \norm{y}_p^{p-1} \left(-(|s| - \varepsilon) \norm{y}_p \right)\nonumber \\
        &\qquad= -(|s| - \varepsilon) \norm{y}^p_p = -p(|s| - \varepsilon) V_{R(\varepsilon,n), p}(z).
    \end{align}
    Since $\varepsilon < |s|$, the rate $\beta \coloneqq  p(|s| - \varepsilon)$ is strictly positive.
\end{proof}


\section{Outer semicontinuity of \texorpdfstring{$G$}{G}}
\label{app:g_semicont}

Take a sequence $z_i = (x_i,\sigma_i) \in \R^n \times \N_N$ converging to $z_\infty = (x_\infty,\sigma_\infty)$. Take $(y_i,\varsigma_i) \in G(x_i,\sigma_i)$ converging to $(y_\infty,\varsigma_\infty)$. We have to prove that $(y_\infty,\varsigma_\infty) \in G(x_\infty,\sigma_\infty)$. 
For every $x\in\R^n$, consider
    $g_k(x) := \inner{\nabla V_{R,p}(x)}{f_k (x)}$,  $\underline{g}(x) := \underset{k \in \N_N}{\min} \, g_k(x)$ 
    and
    $S(x) := \underset{k \in \N_N}{\argmin} \, g_k(x)$.

By~\eqref{eq:maps:g}, it is clear that $y_\infty = x_\infty$. Suppose for a contradiction that $\varsigma_\infty \notin S(x_\infty)$. Since $\varsigma_i\in\N_N$, then there exists $I\in\N$ such that $\varsigma_i = \varsigma_\infty$ for all $i\ge I$. 
The fact that $\varsigma_\infty \notin S(x_\infty)$ implies that $\underline{g}(x_\infty) < g_{\varsigma_\infty}(x_\infty)$. Let $\varepsilon = (g_{\varsigma_\infty}(x_\infty) - \underline{g}(x_\infty))/2 > 0$. The function $\underline{g}$ is continuous because it is the minimum over a finite number of continuous functions. Then, there exists $\delta > 0$ such that $|\underline{g}(x_i) - \underline{g}(x_\infty)| < \varepsilon$ whenever $|x_i - x_\infty| < \delta$, which in turn happens for all $i \ge I_2$, for some $I \le I_2 \in \N$. 
For all $i\ge I_2$, it happens that $\underline{g}(x_i) = g_{\varsigma_i}(x_i) = g_{\varsigma_\infty}(x_i)$. Then, $|g_{\varsigma_\infty}(x_i) - \underline{g}(x_\infty)| < \varepsilon$ holds for all $i\ge I_2$, which contradicts the definition of $\varepsilon$. It follows that $\varsigma_\infty \in S(x_\infty)$ and hence $G$ is outer semicontinuous.

\section{Proofs of Theorems~\ref{thm:eps} and \ref{thm:time_reg}}
\label{app:thm:reg}

This appendix provides the proofs for the space- (Theorem~\ref{thm:eps}) and time-regularized (Theorem~\ref{thm:time_reg}) systems. 
Since both strategies allow the system state to evolve in regions where
$V_{R,p}$ may increase, 
the stability analysis relies on bounding the variation of its time derivative along the subsystems trajectories. 
To this end, we first establish the required bounds through a series of instrumental lemmas.

\subsection{Instrumental Results}

The following results characterize the local continuity of $\nabla V_{R,p}$, establish bounds on the system's flow, and introduce necessary algebraic inequalities.
Let $q \coloneq p/(p-1)$ be the conjugate exponent of $p$. 
First, we establish a local Hölder continuity property for the gradient of $V_{R,p}$.

\medskip
\begin{lemma} \label{lem:holder_grad}
Let $R\in\R^{r\times n}$ full column rank and $p \in (1, \infty)$, and $\alpha \coloneqq \min\{1, p-1\} > 0$. For any $M > 0$, 
\begin{equation}
    \norm{\nabla V_{R,p}(x) - \nabla V_{R,p}(y)}_q \leq L_\nabla(M)  \norm{x - y}_p^\alpha,
\end{equation}
for all $x$, $y \in \R^n$ satisfying $\abs{R_{i:}x}$, $\abs{R_{i:}y}\leq M$ for all $i\in\N_r$, where
\begin{equation}
    L_\nabla(M) \!\coloneqq\!\begin{cases}
         2^{2-p} \norm{R}_p^{p} \!\!\!&\text{if }p\in(1,2], \\
         (p - 1) r^\frac{p-2}{p} \norm{R}_p^2 M^{p-2}, \!\!\!&\text{if }p\in(2,\infty).
    \end{cases}
    \label{eq:L_nabla}
\end{equation}
\end{lemma}
\medskip

\begin{proof}
We first show the local Hölder continuity of $\sigpow{\boldsymbol{\cdot}}{p-1}$. 
For $p\in(1,2]$, from~\cite[Ch.~12]{lindqvist2019notes}, we have
$
\abs{\sigpow{s_1}{p-1}-\sigpow{s_2}{p-1}}\leq 2^{2-p}\abs{s_1-s_2}^{p-1}
$
for all $s_1$, $s_2\in\R$.
Applying this component-wise to vectors $v$, $w\in\R^r$, and using $(p-1)q=p$ and $\alpha = p-1 = \frac{p}{q}$, we obtain
\begin{align}
    &\norm{\sigpow{v}{p-1}-\sigpow{w}{p-1}}_q
        \leq 2^{2-p} \Big(\sum_{i\in\N_r}\abs{v_i-w_i}^{(p-1)q}\Big)^{1/q} \nonumber\\
        &\qquad\qquad\quad = 2^{2-p}\norm{v-w}_p^{\frac{p}{q}}
        = \frac{L_\nabla(M)}{\norm{R}_p^{1+\alpha}}\norm{v-w}_p^\alpha \label{eq:LHolder:1}
\end{align}
for all $v$, $w\in\R^r$ such that $v_i,w_i\in[-M,M]$ for all $i\in\N_r$.

For $p > 2$, we have $\alpha = 1$. Since $\sigpow{s}{p-1}=|s|^{p-2}s$ for all $s\in\R$, 
it follows that $\frac{d}{ds}\sigpow{s}{p-1}=(p-1)|s|^{p-2}\le (p-1)M^{p-2}$ for all $s\in[-M,M]$. Hence, by the mean value theorem,
$\abs{\sigpow{s_1}{p-1}-\sigpow{s_2}{p-1}} \leq (p-1)M^{p-2} \abs{s_1-s_2}$
for all $s_1,s_2\in[-M,M]$.
Applying the previous inequality component-wise to vectors $v$, $w\in\R^r$, and using $\alpha=1$ and
$\norm{\boldsymbol{\cdot}}_q\le r^{\frac{1}{q}-\frac{1}{p}}\norm{\boldsymbol{\cdot}}_p
= r^{\frac{p-2}{p}}\norm{\boldsymbol{\cdot}}_p$, since $p>q>1$, we have
\begin{align}
    &\!\!\! \norm{\sigpow{v}{p-1}-\sigpow{w}{p-1}}_q
            \leq (p-1)M^{p-2}\norm{v-w}_q \nonumber\\
        &\!\!\!\leq \!(p\!-\!1)M^{p-2}r^\frac{p-2}{p}\norm{v-w}_p^\alpha
        = \frac{L_\nabla(M)}{\norm{R}_p^{1+\alpha}} \norm{v-w}_p^\alpha, \label{eq:LHolder:2}
\end{align}
for all $v$, $w\in\R^r$ such that $v_i,w_i\in[-M,M]$ for all $i\in\N_r$.
%
Considering $\nabla V_{R,p}(x)=R^\top\sigpow{Rx}{p-1}$ we have 
\begin{align}
    &\norm{\nabla V_{R,p}(x)-\nabla V_{R,p}(y)}_q = \norm{R^\top \left(\sigpow{Rx}{p-1} - \sigpow{Ry}{p-1}\right)}_q \nonumber\\
        &\leq 
        \! \norm{R}_p \norm{\sigpow{Rx}{p-1} \! - \! \sigpow{Ry}{p-1}}_q \! \label{eq:RTR} \\
        & \le  \norm{R}_p \frac{L_\nabla(M)}{\norm{R}_p^{1+\alpha}}\norm{R(x-y)}_p^\alpha \leq \! L_\nabla(M) \! \norm{x\! -\! y}_p^\alpha, \label{eq:inequality} 
\end{align}
where \eqref{eq:RTR} follows from the induced matrix norm inequality and the identity $\norm{R^\top}_q=\norm{R}_p$
\footnote{
    By duality of $\norm{\boldsymbol{\cdot}}_p$ and $\norm{\boldsymbol{\cdot}}_q$, we have $\norm{z}_p=\max_{\norm{y}_q=1} y^T z$ (see \cite{horn2012matrix}).
    Then, it follows that 
    $\norm{R}_p = \max_{\norm{x}_p=1} \norm{Rx}_p
    = \max_{\norm{x}_p=1}\max_{\norm{y}_q=1} y^\top R x
    = \max_{\norm{y}_q=1}\max_{\norm{x}_p=1} x^\top R^\top y
    = \max_{\norm{y}_q=1}\norm{R^\top y}_q
    = \norm{R^\top}_q$.
    }
, while \eqref{eq:inequality} follows from \eqref{eq:LHolder:1} and \eqref{eq:LHolder:2} with $v=Rx$ and $w=Ry$.
\end{proof}
\medskip

The following lemma bounds the distance between the system's trajectories and their initial conditions.

\medskip
\begin{lemma} \label{lem:flow_bound}
Let $\psi(t)$ be the solution to the subsystem $\dot{\psi} = A \psi + b$ with $\psi(0) = x$. For any $t \geq 0$, we have
\begin{equation}
    \norm{\psi(t) - x}_p \leq \left(\norm{A}_p \norm{x}_p + \norm{b}_p\right) t e^{\norm{A}_p t}. \label{eq:flow_dev}
\end{equation}
\end{lemma}
\medskip

\begin{proof}
Define $\varsigma(t) \coloneqq \psi(t) - x$,
so
    $\dot{\varsigma}(t) = A \psi(t) + b = A \varsigma(t) + A x + b$
with $\varsigma(0)=0$.
Integrating and taking the $p$-norm, we obtain
$\norm{\varsigma(t)}_p  \leq \norm{\int_0^t \left(A \varsigma(s) + A x + b \right)\, \D s }_p
\leq \left(\norm{A}_p \norm{x}_p + \norm{b}_p\right) t + \norm{A}_p \int_0^t \norm{\varsigma(s)}_p \, \D s$.
Then, the result follows from Grönwall's inequality.
\end{proof}
\medskip

The following auxiliary inequality is also needed for the subsequent proofs.

\medskip
\begin{lemma} \label{lem:C_p}
Let $a_1, a_2\geq0$ and $p>0$. Then, 
\begin{equation}
    (a_1+a_2)^p \leq C_p (a_1^p + a_2^p),
\end{equation}
with $C_p \coloneqq \max\{1,\, 2^{p-1}\}$.
\end{lemma}
\medskip

\begin{proof}
For $p\in(0,1]$, the function $(\boldsymbol{\cdot})^p$ is concave and subadditive on $\R_{\geq0}$. Thus, $(a_1+a_2)^p \leq (a_1^p + a_2^p) = C_p (a_1^p + a_2^p)$.
For $p>1$, $(\boldsymbol{\cdot})^p$ is convex. Applying Jensen's inequality, we have $(a_1+a_2)^p \leq 2^{p-1} (a_1^p + a_2^p) = C_p (a_1^p + a_2^p)$.
\end{proof}
\medskip

Using the previous results, we now bound the variation of the time derivative of $V_{R,p}$ along the subsystem trajectories.

\medskip
\begin{lemma} \label{lem:delta_dotV}
Let $x \in \R^n$, $p>1$, and $R\in\R^{r\times n}$ have full column rank. 
Let $\psi(t)$ be the solution to $\dot{\psi} = A \psi + b$ with $\psi(0) = x$, and define $\dot V(t)\coloneqq \inner{\nabla V_{R,p}(\psi(t))}{\dot\psi(t)}$.
There exists $L_{R,p}\in\K_\infty$, dependent on $A$, $b$, $R$, and $p$, but independent of $x$, such that for all $t \geq 0$
\begin{equation}
     \dot V(t) - \dot V(0) \leq L_{R,p}(t) \left(V_{R,p}(x) + 1 \right). \label{eq:delta_dotV}
\end{equation}
\end{lemma}
\medskip

\begin{proof}
For brevity, let $V_x \coloneqq V_{R,p}(x)$, $c_R \coloneqq p^{1/p}\norm{R^\dag}_p$, $c_A \coloneqq \norm{A}_p$, and $c_b \coloneqq \norm{b}_p$. 
From the definition of $V_{R,p}$, $\norm{x}_p = \norm{R^\dag R x}_p\leq \norm{R^\dag}_p \norm{Rx}_p = c_R V_x^{1/p}$. 
We bound the initial derivative of $\psi$ as
\begin{align}
    \norm{\dot\psi(0)}_p &= \norm{A x + b}_p \leq c_A \norm{x}_p + c_b \leq c_A c_R V_x^{\frac{1}{p}} + c_b  \nonumber\\
            & 
            \leq \max\{c_a c_R,c_b\} (V_x^{\frac{1}{p}} +1) \eqqcolon \rho(V_x). \label{eq:rho}
\end{align}
Defining $\gamma(t)\coloneqq t e^{c_A t}$, Lemma~\ref{lem:flow_bound} and \eqref{eq:rho} yield 
\begin{equation}
    \norm{\psi(t) - x}_p \leq \rho(V_x) \gamma(t). \label{eq:delta_psi}
\end{equation}
Using \eqref{eq:rho}, \eqref{eq:delta_psi} and $\dot \psi(0)=Ax+b$, we obtain
\begin{align}
    \norm{\dot\psi(t)}_p  =&\norm{\dot \psi (t)-\dot \psi (0)+\dot \psi (0)}_p \leq \norm{A \left(\psi(t) - x\right)}_p \notag \\
    &+ \norm{\dot\psi(0)}_p
            \leq \rho(V_x) \left(1 + c_A \gamma(t) \right). \label{eq:dotpsi}
\end{align}
By the triangle inequality and \eqref{eq:delta_psi}, we have
\begin{equation}
    \norm{\psi(t)}_p \leq \norm{x}_p + \norm{\psi(t)-x}_p \leq c_R V_x^{1/p} + \rho(V_x)\gamma(t). \label{eq:psi_t}
\end{equation}

We decompose the variation of the derivative, by adding and subtracting $\inner{\nabla V_{R,p}(x)}{\dot{\psi}(t)}$, as
\begin{align}
    \dot V(t) - \dot V(0)
    &= \underbrace{\inner{\nabla V_{R,p}(\psi(t)) - \nabla V_{R,p}(x)}{\dot\psi(t)}}_{\Delta_1} \nonumber\\
    &\quad + \underbrace{\inner{\nabla V_{R,p}(x)}{\dot\psi(t) - \dot\psi(0)}}_{\Delta_2}.
    \label{eq:lem_pert_split}
\end{align}

We first bound  $\Delta_1$. Applying Hölder's inequality and using \eqref{eq:dotpsi} yields
\begin{align}
    &\Delta_1\leq \norm{\nabla V_{R,p}(\psi(t)) - \nabla V_{R,p}(x)}_q \norm{\dot \psi(t)}_p \notag \\
    &\leq  \norm{\nabla V_{R,p}(\psi(t)) - \nabla V_{R,p}(x)}_q \rho(V_x) \left(1 + c_A \gamma(t) \right)
    \label{eq:eq1}
\end{align}
We will apply Lemma~\ref{lem:holder_grad} to bound \eqref{eq:eq1}. Note that
\begin{align}
\abs{R_{i:}\psi(t)}&\leq  \norm{R\psi(t)}_\infty \leq \norm{R\psi(t)}_p \leq \norm{R}_p  \norm{\psi(t)}_p \notag \\
&\leq \norm{R}_p \Big( c_R V_x^{\frac{1}{p}} + \rho(V_x)\gamma(t) \Big)=:M(t) \label{eq:M}
\end{align}
where in \eqref{eq:M} we have used \eqref{eq:psi_t}.
Let $\alpha \coloneqq \min\{1, p-1\}$.
Then, the hypotheses of Lemma~\ref{lem:holder_grad} are satisfied with $M(t)$ for each pair
$\psi(t),x \in \R^n$ which yields
\begin{align*}
&\norm{\nabla V_{R,p}(\psi(t)) - \nabla V_{R,p}(x)}_q \notag \\
&\le L_\nabla(M(t)) \norm{\psi(t) - x}_p^\alpha 
    \le L_\nabla(M(t)) \rho(V_x)^\alpha\gamma(t)^\alpha,
\end{align*}
where we have used \eqref{eq:delta_psi}.
%
From 
the previous inequality
and \eqref{eq:eq1}, we have
\begin{equation}
    \Delta_1 \leq L_\nabla(M(t)) \rho(V_x)^{1+\alpha} \gamma(t)^\alpha \bigl(1 + c_A \gamma(t)\bigr). \label{eq:lem_pert_T1_base}
\end{equation}
Since $1+\alpha\ge 1$, according to Lemma~\ref{lem:C_p} we can bound
\begin{align}
&\rho(V_x)^{1+\alpha}
=
\max\{c_A c_R,c_b\}^{1+\alpha} (V_x^{\frac{1}{p}} + 1)^{1+\alpha} \notag \\
&\leq
 C_\rho \Bigl( V_x^{\frac{1+\alpha}{p}} + 1\Bigr) 
\label{eq:rho2}
\end{align}
with $C_\rho:=2^\alpha \max\{(c_A c_R)^{1+\alpha},c_b^{1+\alpha}\}$.

%
We consider two cases depending on the value of $p$. 

For $p\in(1,2]$, we have  $\alpha=p-1$, and from \eqref{eq:L_nabla}, $L_\nabla(M(t))=2^{2-p} \norm{R}_p^{p}$.
Then  
using \eqref{eq:rho2}, we obtain
\begin{equation}
   \!\! L_\nabla(M(t))  \rho(V_x)^{1+\alpha}\leq K_1 (V_x+1) \gamma(t)^{\alpha} \bigl(1 + c_A \gamma(t)\bigr).
    \label{eq:Lrho1}
\end{equation}
with $K_1:=2 \norm{R}_p^{p}  \max\{ (c_A c_R)^{p} ,  c_b^{p}\}$.

For $p>2$, we have $\alpha=1$ and 
$L_\nabla(M(t)) = C_\nabla M(t)^{p-2}$ with constant $C_\nabla > 0$ given by \eqref{eq:L_nabla}.
From \eqref{eq:M} and \eqref{eq:rho}, defining $c_M {\coloneqq} \norm{R}_p \max\{c_R,\, c_A c_R, c_b\}$ we obtain $M(t) \leq c_M \Big(V_x^\frac{1}{p}+1\Big)(1+\gamma(t))$.
Applying Lemma~\ref{lem:C_p}, there exists a constant $C_{p-2} > 0$ such that
\begin{align*}
    &L_\nabla (M(t))\rho(V_x)^{1+\alpha} =C_\nabla M(t)^{p-2}\rho(V_x)^{1+\alpha} \notag \\
    &\leq  C_\nabla c_M^{p-2} C_{p-2} \Big(V_x^\frac{p-2}{p} + 1\Big)(1+\gamma(t))^{p-2} \rho(V_x)^{1+\alpha}.
\end{align*}
Defining $K_2 \coloneqq C_\nabla c_M^{p-2} C_{p-2} C_\rho$ and using \eqref{eq:rho2} we have
\begin{align}
    &L_\nabla(M(t))\rho(V_x)^{1+\alpha} \nonumber\\
    &\qquad\leq K_2 \Big(V_x^\frac{p-2}{p} + 1\Big) \Big(V_x^\frac{2}{p}+1\Big) (1+\gamma(t))^{p-2} \nonumber\\
    &\qquad =  K_2 \Big(V_x + V_x^\frac{p-2}{p} + V_x^\frac{2}{p} + 1 \Big) (1+\gamma(t))^{p-2} \label{eq:inter} \\
    &\qquad \leq 3 K_2 \Big(V_x + 1 \Big) (1+\gamma(t))^{p-2},  \label{eq:Lrho2}
\end{align}
where in \eqref{eq:inter} we have used that $V_x^{\frac{p-2}{p}} \le  V_x+1$ and $V_x^{\frac{2}{p}} \le  V_x+1$
since $\frac{p-2}{p}<1$ and $\frac{2}{p}<1$.

Define $L_1(t)\coloneqq \max\{K_1,\, 3K_2\}(1+\gamma(t))^{\max\{0,\,{p-2}\}}$ which is continuous, non-decreasing, and non-negative. From \eqref{eq:Lrho1}  and \eqref{eq:Lrho2}, we have $L_\nabla(M(t))\rho(V_x)^{1+\alpha}\leq L_1(t)(V_x+1)$.
Using this into \eqref{eq:lem_pert_T1_base}, we obtain
\begin{equation}
    \Delta_1 \leq L_1(t) \left(V_0 + 1\right) \gamma(t)^\alpha (1 + c_A \gamma(t)). \label{eq:T1_final}
\end{equation} 

Next, we bound $\Delta_2$. From the gradient definition, we have 
\begin{align}
    \norm{\nabla V_{R,p}(x)}_q &= \norm{R^\top \sigpow{Rx}{p-1}}_q \leq \norm{R^\top}_q \norm{\sigpow{Rx}{p-1}}_q \nonumber\\
    & 
    \leq \norm{R}_p^p \norm{x}_p^{p-1} \!\leq \norm{R}_p^p c_R^{p-1} V_x^\frac{p-1}{p}. \label{eq:grad_bound}
\end{align} 
Applying Hölder's inequality to $\Delta_2$, and using \eqref{eq:delta_psi} and \eqref{eq:grad_bound}, yields
\begin{align}
   \Delta_2 &\leq \norm{\nabla V_{R,p}(x)}_q \norm{A(\psi(t) - x)}_p \nonumber\\
       &\leq \norm{R}_p^p c_R^{p-1} V_x^{\frac{p-1}{p}} c_A \rho(V_x) \gamma(t)\nonumber\\
       &= \norm{R}_p^p c_R^{p-1} c_A \Big(c_A c_R V_x + c_b V_x^{\frac{p-1}{p}} \Big) \gamma(t).
\end{align}
Given that $\frac{p-1}{p}<1$, using the bound $V_x^{\frac{p-1}{p}} \leq V_x + 1$ directly establishes 
    $\Delta_2 \leq L_2 \left(V_x + 1\right) \gamma(t)$, 
for an appropriate constant $L_2>0$.
Combining \eqref{eq:T1_final} and 
 this bound on $\Delta_2$, we define
\begin{equation}
    L_{R,p}(t) \coloneqq L_1(t) \gamma(t)^\alpha \left(1 + c_A\gamma(t)\right) + L_2 \gamma(t).
\end{equation}
Since $\gamma(0) = 0$ and $\gamma(t)$ is continuous and strictly increasing for $t \geq 0$, and satisfies $\lim_{t \to \infty} \gamma(t) = \infty$, we have $L_{R,p}\in\K_\infty$, which completes the proof.
\end{proof}

\medskip
\begin{lemma} \label{lem:convex_p_ineq}
Let $p \in (1, \infty)$. For all $a_1$, $a_2 \geq 0$ and $s > 0$, the following inequality holds:
\begin{equation}
    (a_1+a_2)^p \leq (1+s)^{p-1} a_1^p + \Big(\frac{1+s}{s}\Big)^{p-1} a_2^p.
\end{equation}
\end{lemma}
\medskip

\begin{proof}
Define $\theta\coloneq\frac{1}{1+s }\in(0,1)$.
Since $p \in (1, \infty)$, the function $a\mapsto a^p$ is convex on $\R_{\geq 0}$. By convexity, we have
$(a_1 + a_2)^p = \left(\theta \frac{a_1}{\theta} + (1-\theta) \frac{a_2}{1-\theta}\right)^p 
        \leq \theta \left(\frac{a_1}{\theta}\right)^p + (1-\theta) \left(\frac{a_2}{1-\theta}\right)^p = \theta^{1-p} a_1^p + (1-\theta)^{1-p} a_2^p$. 
Substituting $\theta=\frac{1}{1+s}$ and $1-\theta = \frac{s}{1+s}$, the result follows.
\end{proof}


\subsection{Proof of Theorem~\ref{thm:eps} (Space Regularization)}

\medskip
\noindent
\ref{thm:eps:dwell})
Let $K_{R,p,k}\in\K_\infty$ for each $k\in \N_N$ be
the 
$L_{R,p}$ from Lemma~\ref{lem:delta_dotV} applied to 
$\dot x = f_k(x)$, and define
$K_{R,p}(t) \coloneqq \max_{k\in\N_N}K_{R,p,k}(t) \in \K_\infty$.
Since $R$ and $S$ have full column rank, there exist  $c_V$, $c_R > 0$ such that $V_{S,p}(x) \geq c_V V_{R,p}(x)$, and $\norm{x}_p \leq c_R V_{R,p}(x)^{1/p}$ for all $x \in \R^n$.

Consider $(x,\sigma)\in\setD_\varepsilon$ and let $[x^{+\top},\sigma^+]^\top\in G(x,\sigma)$.
By definition of $\setD_\varepsilon$, $V_{R,p}(x) \geq \varepsilon$. 
Let $\psi(t)$ denote the solution to $\dot{x} = A_{\sigma^+} x + b_{\sigma^+}$ with $\psi(0) = x^+ = x$, and define $\dot V(t)\coloneqq \inner{\nabla V_{R,p}(\psi(t))}{\dot\psi(t)}$.
By Lemma~\ref{lem:min_k} and the definition of $G$, if follows that $\dot V(0) \leq -V_{S,p}(x)<-\eta V_{S,p}(x)$ which implies $(x,\sigma^+)\in\setC_\varepsilon\backslash\setD_\varepsilon$. 
By definition of $\setD_\varepsilon$, a subsequent jump cannot occur as long as the trajectory satisfies
\begin{equation}
    \delta(t) \coloneqq \dot V(t) + \eta V_{S,p}(\psi(t)) < 0. \label{eq:ineq_flow}
\end{equation}
Continuity of the left-hand side of \eqref{eq:ineq_flow} and the fact that $\dot V(0) \leq -V_{S,p}(x)$ ensure that \eqref{eq:ineq_flow} holds over a sufficiently small time interval.
However, establishing a minimum dwell-time requires a uniform condition valid for all $(x,\sigma) \in \setD_\varepsilon$. 
To this end, we bound the left-hand side of \eqref{eq:ineq_flow} by applying Lemma~\ref{lem:delta_dotV}, which for all $t \geq 0$ yields
%
\begin{align}
    \dot V(t) &\leq \dot V(0) + K_{R,p}(t)(V_{R,p}(x) + 1) \nonumber\\
    &\leq -V_{S,p}(x) + K_{R,p}(t)(V_{R,p}(x) + 1). \label{eq:eps_dotV_bound}
\end{align}
Applying the triangle inequality and Lemma~\ref{lem:convex_p_ineq} with $s\coloneqq t>0$, for all $t>0$ we have 
\begin{align}
    &\!\!\!V_{S,p}(\psi(t)) \leq p^{-1} \left( \norm{Sx}_p  + \norm{S}_p \norm{\psi(t)-x}_p\right)^p\nonumber\\
    &\!\!\!\leq \! (1\!+t)^{p-1} V_{S,p}(x) \!+\! \left(\!\frac{1+t}{t}\!\right)^{p-1} \! \frac{\norm{S}_p^p}{p} \! \norm{\psi(t)\!-\!x}_p^p. \label{eq:eps_VS_bound}
\end{align}
Let $c_A \coloneqq \max_{i\in\N_N} \norm{A_i}_p$, and $c_b \coloneqq \max_{i\in\N_N} \norm{b_i}_p$. 
From Lemma~\ref{lem:flow_bound} and the definition of $c_R$, we have $\norm{\psi(t)-x}_p \leq (c_A c_R V_{R,p}(x)^{1/p}+ c_b) t e^{c_A t}$.
Substituting this into \eqref{eq:eps_VS_bound}, and using Lemma~\ref{lem:C_p}, we obtain
\begin{align}
    &V_{S,p}(\psi(t)) \leq (1+t)^{p-1} V_{S,p}(x) \nonumber\\
    &\quad + \left(1+t\right)^{p-1} \frac{\norm{S}_p^p}{p} (c_A c_R V_{R,p}(x)^{1/p}+ c_b)^p t \, e^{p c_A t} \nonumber\\
    &\leq \!\! \biggl(\!(1+t)^{p-1} \!+ \!\underbrace{(1+t)^{p-1} t \frac{C_p}{p} \! \left(\norm{S}_p c_A c_R e^{c_A t}\right)^p\!}_{K_S(t)} \, \biggr) V_{S,p}(x) \nonumber\\
    &\quad+ \underbrace{(1+t)^{p-1} t \frac{C_p}{p} \left(\norm{S}_p c_b e^{c_A t}\right)^p}_{\alpha_S(t)}, \label{eq:eps_VS_final}
\end{align}
where $K_S$, $\alpha_S \in \K_\infty$.
Combining \eqref{eq:eps_dotV_bound} and \eqref{eq:eps_VS_final}, we have
\begin{equation*}
    \delta(t)  \leq  K_{R,p}(t)V_{R,p}(x) - \vartheta(t) V_{S,p}(x) + \alpha(t) ,
\end{equation*}
where 
$\vartheta(t)\coloneqq 1 -\eta (1+t)^{p-1} - \eta K_S(t)$, and $\alpha(t)\coloneqq K_{R,p}(t) + \eta \alpha_S(t) \in\K_\infty$.
Since $K_S \in \K_\infty$ and $\eta \in (0,1)$, $\vartheta(t)$ is continuous, strictly decreasing, and satisfies $\vartheta(0) = 1-\eta > 0$.
By continuity, $\vartheta(t)>0$ for all $t \in [0, T_\vartheta]$
for some $T_\vartheta > 0$.
Using $V_{S,p}(x) \geq c_V V_{R,p}(x)$, for all $t \in [0,T_\vartheta]$ we obtain
\begin{equation}
    \delta(t) \! \leq \! \left(K_{R,p}(t) - \vartheta(t) c_V\right)V_{R,p}(x) + \alpha(t) \eqqcolon \Delta(x, t).
\end{equation} 
Given that $V_{R,p}(x) \geq \varepsilon$, evaluating $\Delta$ at $t=0$ yields
\begin{equation}
    \Delta(x, 0) = -c_V (1-\eta)V_{R,p}(x) \leq -c_V(1-\eta)\varepsilon < 0.
\end{equation}
By continuity of $\Delta(x,t)$ with respect to $t$, there exists a uniform constant $T(\varepsilon) \in (0, T_\vartheta]$ such that $\Delta(x, t) < 0$ for all $t \in [0, T(\varepsilon)]$ and for all $(x,\sigma)\in\setD_\varepsilon$. 
Consequently, \eqref{eq:ineq_flow} holds for all $t \in [0, T(\varepsilon)]$ and the result follows.

\medskip
\noindent
\ref{thm:eps:UGAS})
Let $z \coloneqq (x,\sigma) \in \R^n \times \N_N$ denote the hybrid state. Define $V_\varepsilon(x) \coloneqq V_{R,p}(x) - \varepsilon$ and $W(z) \coloneqq \max\{V_\varepsilon(x),\, 0\}$.
Since $R$ has full column rank, $V_{R,p}$ is positive definite and radially unbounded, implying that $\setA_\varepsilon$ is compact. Furthermore, $W(z)$ is continuous, positive definite with respect to $\setA_\varepsilon$, and radially unbounded. Thus, there exist $\underline\alpha$, $\overline\alpha \in \K_\infty$ such that $\underline\alpha(\norm{z}_{\setA_\varepsilon}) \leq W(z) \leq \overline\alpha(\norm{z}_{\setA_\varepsilon})$ for all $z$. 
Since $R$ and $S$ have full column rank, there exists a constant $c_V > 0$ such that $V_{S,p}(x) \geq c_V V_{R,p}(x)$ for all $x \in \R^n$.

Let $\phi(t,j) = (x(t,j), \sigma(t,j))$ be a solution to $\Hy_\varepsilon$. 
Pick any $(t, j) \in \dom{\phi}$ and let $0 = t_0 \leq t_1 \leq \dots \leq t_{j+1} = t$ satisfy
\begin{equation}
    \dom{\phi} \cap \left( [0, t] \times \{0, . . . , j\}\right) = \! \bigcup_{i\in\{0, \dots, j\}} \!\! [t_i,\, t_{i+1}] \times\{i\}.
\end{equation}

\textit{Step 1: Forward pre-invariance of $\setA_\varepsilon$.}
We first show that any solution starting in $\mathcal{A}_\varepsilon$ remains there for its duration.
Assume that $V_{R,p}(x(0,0)) \leq \varepsilon$. 
Suppose, for a contradiction, that the trajectory 
escapes $\setA_\varepsilon$.
Since jumps keep $x$ constant, the trajectory must escape $\setA_\varepsilon$ during flow.
Since $V_{R,p}(x(t,j))$ is continuous, there exist a hybrid time interval $[s_1, s_2]\times\{i\}\in\dom{\phi}$ such that $V_{R,p}(x(s_1,i)) = \varepsilon$ and $V_{R,p}(x(s,i)) > \varepsilon$ for all $s \in (s_1, s_2]$.
Thus, by definition of $\setC_\varepsilon$, we have that $V_{R,p}(x(s,i))$ is strictly decreasing for all $s \in (s_1, s_2]$. 
This contradicts $V_{R,p}(x(s_2,i)) > \varepsilon = V_{R,p}(x(s_1,i))$, and consequently, $V_{R,p}(x(t,j)) \leq \varepsilon$ for all $(t, j)\in\dom{\phi}$.

\textit{Step 2: Bounding $W$.}
Let $(t^*,j^*) \in \dom{\phi}$ be the first hybrid time such that $\phi(t^*, j^*) \in \setA_\varepsilon$ (if the trajectory never reaches $\setA_\varepsilon$, we set $t^* = t$, $j^* = j$).
For all $(s,i)\in\dom{\phi}$ such that $s<t^*$, we have $V(x(s,i))>\varepsilon$. 
By the definitions of $\setC_\varepsilon$ and $c_V$, during flows we have
$\inner{\nabla V_{R,p}(x(s,i))}{f_{\sigma(x,i)}x(s,i)} \leq - \eta V_{S,p} (x(s,i))  
\leq - \eta c_V V_{R,p}(x(s,i))$.

Applying the Comparison Lemma \cite[Lem. 3.4]{khalil2002nonlinear} on each flow interval $[t_i, t_{i+1}]$ with $i<j^*$, we obtain
\begin{equation}
    V_{R,p}(x(t_{i+1},i)) \leq e^{-\eta c_V (t_{i+1}-t_i)} V_{R,p}(x(t_i,i)). \label{eq:bound_ti}
\end{equation}
Analogously, applying the Comparison Lemma \cite[Lem. 3.4]{khalil2002nonlinear} on $[t_{j^*}, t^*]$, we have
\begin{equation}
    V_{R,p}(x(t^*,j^*)) \leq e^{-\eta c_V (t^*-t_{j^*})} V_{R,p}(x(t_{j^*},j^*)). \label{eq:bound_tj}
\end{equation}
Since $x$ does not change during jumps, $V_{R,p}(x(t_{i}, i)) = V_{R,p}(x(t_{i}, i-1))$ for all $i=\{1,\dots,j\}$.
By concatenating \eqref{eq:bound_ti} and \eqref{eq:bound_tj}, we obtain
\begin{equation}
    V_{R,p}(x(t^*,j^*)) \leq e^{-\eta c_V t^*} V_{R,p}(x(0,0)).
\end{equation}
Subtracting $\varepsilon$ from both sides, using the fact that $e^{-\eta c_V t^*} \leq 1$ and the definition of $W$, we have
\begin{align*}
    &V_\varepsilon(x(t^*,j^*))  \leq e^{-\eta c_V t^*} V_{R,p}(x(0,0)) - \varepsilon \nonumber \\
    &\leq e^{-\eta c_V t^*} \left( V_{R,p}(x(0,0)) - \varepsilon \right) \leq e^{-\eta c_V t^*} W(\phi(0,0)). \label{eq:V_eps_global}
\end{align*}
Since the right-hand side is nonnegative, it follows that $W(\phi(t^*,j^*)) \leq e^{-\eta c_V t^*} W(\phi(0,0))$.
For all $(t,j) \in \dom{\phi}$ such that $t \geq t^*$, the forward pre-invariance established in Step 1 guarantees $W(\phi(t,j)) = 0\leq W(\phi(t^*,j^*))$. 
Thus, for all $(t,j) \in \dom\phi$, we conclude
\begin{equation}
    W(\phi(t,j)) \leq e^{-\eta c_V t} W(\phi(0,0)). \label{eq:W_global}
\end{equation}

\textit{Step 3: The $\KL$ bound.}
By item \ref{thm:eps:dwell}), we have $t \geq (j-1)T(\varepsilon)$ for all $(t,j) \in \dom \phi$.
Then, $t + j \leq t \left(1 + \frac{1}{T(\varepsilon)}\right) + 1$, which implies $t \geq \frac{T(\varepsilon)}{T(\varepsilon)+1} (t+j-1)$. Substituting in \eqref{eq:W_global} yields
    $W(\phi(t,j)) \leq e^{-\eta c_V \frac{T(\varepsilon)}{T(\varepsilon)+1} (t+j-1)} W(\phi(0,0)) 
    = C_\varepsilon e^{-\mu (t+j)} W(\phi(0,0))$,
where $\mu \coloneqq \eta c_V \frac{T(\varepsilon)}{T(\varepsilon)+1} > 0$ and $C_\varepsilon \coloneqq e^{\mu} > 1$.
Using the functions $\underline\alpha$ and $\overline\alpha$, we obtain
\begin{equation}
    \norm{\phi(t,j)}_{\setA_\varepsilon} \leq \underline\alpha^{-1}\left( C_\varepsilon e^{-\mu (t+j)} \overline\alpha(\norm{\phi(0,0)}_{\setA_\varepsilon}) \right).
\end{equation}
Defining the function $\beta(r, s) \coloneqq \underline\alpha^{-1}\left( C_\varepsilon e^{-\mu s} \overline\alpha(r) \right) \in \KL$, we have $\norm{\phi(t,j)}_{\setA_\varepsilon} \leq \beta(\norm{\phi(0,0)}_{\setA_\varepsilon}, t+j)$ for all $(t,j) \in \dom \phi$.
By \cite[Theorem 3.40]{goebel2012hybrid}, this $\KL$ bound implies that $\setA_\varepsilon$ is UGpAS for $\Hy_\varepsilon$.
Finally, as established in the proof of Theorem~\ref{th:UGAS}, the hybrid basic conditions \cite[Ass. 6.5]{goebel2012hybrid} hold and maximal solutions cannot escape to infinity in finite time due to the affine dynamics.
Therefore, maximal solutions are complete, and UGpAS implies UGAS of $\setA_\varepsilon$ for $\Hy_\varepsilon$. \hfill $\blacksquare$


\subsection{Proof of Theorem~\ref{thm:time_reg} (Time Regularization)}
\label{app:thm:time_reg}

\noindent
\ref{thm:time_reg:dwell})
By~\eqref{eq:maps_time}, the timer $\tau$ satisfies $\tau^{+}=0$ after each jump. Since $\dot\tau = \min \{1,2- \tau/T\}=1$ for $\tau\in[0,T]$, the timer takes exactly $T$ time units to rise from $0$ to $T$. Given that the jump set $\setD_T$ in~\eqref{eq:setD_T} requires $\tau\geq T$, no jump can occur within $T$ time units of the previous one, giving $t\geq(j-1)T$ for all $(t,j)\in\dom\phi$.

\medskip
\noindent
\ref{thm:time_reg:UGAS})
For each $k\in\N_N$, let $K_{R,p,k}\in\K_\infty$ be the function $L_{R,p}$ from Lemma~\ref{lem:delta_dotV} applied to the subsystem $\dot x = f_k(x)$, and define
$K_{R,p}(t) \coloneqq \max_{k\in\N_N}K_{R,p,k}(t) \in \K_\infty$.
Since $R$ and $S$ have full column rank, there exist constants $c_V$, $c_R > 0$ such that 
$V_{S,p}(x) \geq c_V V_{R,p}(x)$, and $\norm{x}_p \leq c_R V_{R,p}(x)^{1/p}$ for all $x \in \R^n$.
Define $T_M \coloneqq \min\{K_{R,p}^{-1}(\frac{c_V}{2}),\, \frac{1}{2c_V} \}> 0$ and consider $T\in[0,T_M]$. Then, we have $K_{R,p}(T) \leq \frac{c_V}{2}$. 
Define $\gamma(T) \coloneqq c_V - K_{R,p}(T) \in \left[\frac{c_V}{2}, c_V\right]$.
By definition of $T_M $, 
\begin{equation}
    1 - T\gamma(T) \geq 1 - T_M  \max_{T\in[0,T_M ]} \gamma(T) = 1 - T_M  c_V \geq \frac{1}{2}.
    \label{eq:TgammaT}
\end{equation}
Let $\varepsilon(T) \coloneqq \frac{K_{R,p}(T)}{\gamma(T)} > 0$,
which is strictly increasing and satisfies $\lim_{T\to0^+} \varepsilon(T) = 0$.
Define $V_\varepsilon(x)\coloneqq V_{R,p}(x)-\varepsilon(T)$,
and $W(x)\coloneqq\max\{0,V_\varepsilon(x)\}$.

Let $\phi(t,j) = (x(t,j), \sigma(t,j), \tau(t,j))$ be a solution to $\Hy_T$, with initial condition $\phi(0,0)= (x_0,\, \sigma_0,\, \tau_0) \in \setC_T \cup \setD_T$. 
Pick any $(t, j) \in \dom{\phi}$ and let $0 = t_0 \leq t_1 \leq \dots \leq t_{j+1} = t$ satisfy
    $\dom{\phi} \, \cap \, \left( [0, t] \times \{0, . . . , j\}\right) = \! \bigcup_{i\in\{0, \dots, j\}} [t_i,\, t_{i+1}] \times\{i\}$.
Each flow interval $[t_i,\,t_{i+1}]$ can be separated into two phases: the forced flow, where the timer $\tau < T$ prevents switching, and the regular flow, where $\tau \geq T$ and the system flows as long as the gradient condition in \eqref{eq:setC} is satisfied.

\textit{Step 1: Bounding the forced flow.}
Consider the time interval $[0,t_1]$. 
Define $\Delta_0  \coloneqq \max\{0, T - \tau_0\} \leq T$
and suppose $\tau_0<T$, so that $\Delta_0 >0$. 
During the time interval $[0, \Delta_0 ]$, the system is forced to flow regardless of the gradient condition, and the value of $V_{R,p}$ may increase. 
Using the triangle inequality and the submultiplicativity of $\norm{\boldsymbol{\cdot}}_p$, we have $\norm{Rx(\Delta_0 ,0)}_p \leq \norm{Rx_0}_p + \norm{R}_p \norm{x(\Delta_0 ,0) - x_0}_p$. Applying Lemma~\ref{lem:convex_p_ineq} with $s\coloneqq\Delta_0$, it follows that
\begin{align}
    &\!\!\!V_{R,p}(x(\Delta_0 ,0)) \! \leq p^{-1} \! \left(\norm{Rx_0}_p + \norm{R}_p \norm{x(\Delta_0 ,0) - x_0}_p\right)^p \nonumber\\
    &\quad \leq (1+\Delta_0)^{p-1} V_{R,p}(x_0) \nonumber\\
    &\qquad + \left(\frac{1+\Delta_0}{\Delta_0}\right)^{p-1} \frac{\norm{R}_p^p}{p} \norm{x(\Delta_0 ,0) - x_0}_p^p. \label{eq:V_T}
\end{align}
Let $c_A \coloneqq \max_{i\in\N_N} \norm{A_i}_p$, and $c_b \coloneqq \max_{i\in\N_N} \norm{b_i}_p$. 
By Lemma~\ref{lem:flow_bound}, and using the definition of $c_R$, we have $\norm{x(\Delta_0 ,0) - x_0}_p \leq (c_A c_R V_{R,p}(x_0)^{\frac{1}{p}} + c_b) \Delta_0  e^{c_A \Delta_0 }$. 
Using this in \eqref{eq:V_T}, and applying Lemma~\ref{lem:C_p}, we obtain
{\allowdisplaybreaks
\begin{align}
    &V_{R,p}(x(\Delta_0 ,0)) \leq (1+\Delta_0)^{p-1} V_{R,p}(x_0) \nonumber\\
    &\; + \! \left(\!\frac{1+\Delta_0}{\Delta_0}\!\!\right)^{p-1} \! \frac{\norm{R}_p^p}{p} \!\left( c_A c_R V_{R,p}(x_0)^{\frac{1}{p}} + c_b \right)^p \!\!\left(\Delta_0  e^{c_A \Delta_0 }\right)^p \nonumber\\
    &\leq \!\underbrace{(1+\Delta_0 )^{p-1} \! \left(\!1 \!+\! \Delta_0  \frac{C_p}{p}\!\left(\norm{R}_p c_A c_R e^{c_A \Delta_0 }\right)^p \right)}_{\hat\gamma(\Delta_0 )} V_{R,p}(x_0) \nonumber\\
    &\; + \underbrace{(1+\Delta_0 )^{p-1} \Delta_0 \frac{C_p}{p} \left(\norm{R}_p c_b   e^{c_A \Delta_0 }\right)^p}_{\alpha_0(\Delta_0 )}. \label{eq:V_t1}
\end{align}
}
Note that $\alpha_0 \in \K_\infty$, and that $\hat\gamma$ is strictly increasing, and $\lim_{\Delta_0 \to0^+} \hat\gamma(\Delta_0 ) = 1$.
On the other hand, if $\tau_0 \geq T$, we have $\Delta_0 =0$, there is no forced flow and~\eqref{eq:V_t1} holds trivially with $\hat\gamma(\Delta_0 ) \coloneqq 1$ and $\alpha_0(\Delta_0 ) = 0$.
Subtracting $\varepsilon(T)$ from both sides of \eqref{eq:V_t1}, using the definitions of $V_\varepsilon$ and $W$, and the fact that $\hat\gamma(\Delta_0 )\leq \hat\gamma(T)$, we obtain
\begin{align}
    &\!\!\!\!V_\varepsilon(x(\Delta_0 ,0)) \leq \hat\gamma(\Delta_0 ) V_\varepsilon(x_0) + \alpha_0(\Delta_0 ) + \varepsilon(T) (\hat\gamma(\Delta_0 )-1), \notag\\
        & \leq \hat\gamma(T) V_\varepsilon(x_0) + \hat\alpha(\Delta_0 ) 
         \leq \hat\gamma(T) W(x_0) + \hat\alpha(\Delta_0 ), \label{eq:V_T0}
\end{align}
where $\hat\alpha(T)\in\K_\infty$ is chosen such that $\hat\alpha(\Delta_0 ) \geq \alpha_0(\Delta_0 ) + \varepsilon(T) (\hat\gamma(\Delta_0 )-1)$, which is possible since $\hat\gamma$ is strictly increasing, $\lim_{\Delta_0 \to0^+} \hat\gamma(\Delta_0 ) = 1$, and $\alpha_0 \in \K_\infty$.

Let us now consider the time intervals $[t_i,t_{i+1}]$ for all $i\in[1,j]$.
By Lemma~\ref{lem:min_k} and the definition of $G_T$, $\sigma_i\coloneqq\sigma(t_i,i)$ is chosen such that
\begin{align}
    &\inner{\nabla V_{R,p}(x(t_i,i))}{f_{\sigma_i}(x(t_i,i))} \leq -V_{S,p}(x(t_i,i)) \nonumber\\
    & \qquad\qquad\qquad\qquad\qquad\qquad\leq -c_V V_{R,p}(x(t_i,i)).
    \label{eq:init_deriv_i}
\end{align}
From item \ref{thm:time_reg:dwell}), we have $t_{i+1} - t_i \geq T$. Let $s\in[t_i,t_i+T)$, thus $\tau(s,i) < T$ and the system is forced to flow regardless of the gradient condition.
Applying the triangle inequality, Lemma \ref{lem:delta_dotV} to subsystem $\sigma_i$, and using \eqref{eq:init_deriv_i}, we obtain
\begin{align}
    &\inner{\nabla V_{R,p}(x(s,i))}{f_{\sigma_i} (x(s,i))} \nonumber\\
    &\leq  - c_V V_{R,p}(x(t_i,i)) + K_{R,p}(s-t_i)(V_{R,p}(x(t_i,i)) + 1) \nonumber\\
    & \leq - \left(c_V - K_{R,p}(T)\right)V_{R,p}(x(t_i,i)) + K_{R,p}(T) \nonumber\\
    & = - \gamma(T) \left(V_{R,p}(x(t_i,i)) - \varepsilon(T) \right)  = - \gamma(T) V_\varepsilon(x(t_i,i)), \label{eq:dotV_forced}
\end{align}
where we used the definitions of $\gamma(T)$ and $\varepsilon(T)$.
Noting that the time derivative of $V_\varepsilon$ equals that of $V_{R,p}$, integrating \eqref{eq:dotV_forced} on the interval $[t_i, t_i+T]$ gives 
\begin{align}
    V_\varepsilon(x(t_i+T,i)) &\leq \left(1 - T \gamma(T)\right) V_\varepsilon(x(t_i,i)) \nonumber\\
    &\leq \left(1 - T \gamma(T)\right) W(x(t_i,i)), \label{eq:V_i}
\end{align}
where we used the definition of $W$ and the fact that, from \eqref{eq:TgammaT}, $1 - T \gamma(T)>0$.

\textit{Step 2: Bounding the regular flow.}
Consider 
   $s \in [\Delta_0 , t_1] \cup \bigcup_{i\in\{1, \dots, j\}} \!\! [t_i+T,\, t_{i+1}]$,
thus $\tau(s,0) \geq T$ and the timer does not prevents switching. 
Then, the solution flows satisfying the gradient condition, implying
\begin{align}
    &\inner{\nabla V_{R,p}(x(s,0))}{f_{\sigma_0}(x(s,0))} \leq -\eta V_{S,p}(x(s,0)) \nonumber\\
    &\qquad\qquad\qquad\qquad\qquad\qquad\leq - \eta c_V V_{R,p}(x(s,0)).
    \label{eq:init_deriv_0}
\end{align}
Applying the Comparison Lemma \cite[Lem. 3.4]{khalil2002nonlinear} on $[\Delta_0 ,t_1]$, we obtain
$V_{R,p}(x(t_1,0)) \leq e^{-\eta c_V(t_1-\Delta_0 )} V_{R,p}(x(\Delta_0 ,0))$, and 
\begin{align}
    V_\varepsilon(x(t_1,0)) &\leq e^{-\eta c_V(t_1-\Delta_0 )} V_{R,p}(x(\Delta_0 ,0)) - \varepsilon(T) \nonumber\\
        &\leq e^{-\eta c_V(t_1-\Delta_0 )} V_\varepsilon(x(\Delta_0 ,0)), \label{eq:V_0_reg}
\end{align}
where we used $e^{-\eta c_V(t_1-\Delta_0 )}\leq 1$.
Analogously, applying the Comparison Lemma on each interval $[t_i+T,t_{i+1}]$ for $i\in\{1,\dots,j\}$, we obtain
\begin{align}
    V_\varepsilon(x(t_{i+1},i)) \leq e^{-\eta c_V(t_{i+1}-t_i-T)} V_\varepsilon(x(t_i+T,i)). \label{eq:V_i_reg}
\end{align}

\textit{Step 3: the $\KL$ bound.}
Let $\mu \coloneqq \min\{\frac{c_V}{2},\, \eta c_V\}>0$, so that $\gamma(T)\geq\frac{c_V}{2}\geq\mu$ and $\eta c_V\geq \mu$.
From \eqref{eq:V_T0} and \eqref{eq:V_0_reg}:
\begin{equation}
    V_\varepsilon(x(t_1,0)) \leq e^{-\mu(t_1-\Delta_0 )} \left(\hat\gamma(T) W(x_0) + \hat\alpha(\Delta_0 )\right).
\end{equation}
Given that the right-hand side is nonnegative, from the definition of $W$ it follows that
\begin{equation}
    W(x(t_1,0)) \leq e^{-\mu(t_1-\Delta_0 )} \left(\hat\gamma(T) W(x_0) + \hat\alpha(\Delta_0 )\right). \label{eq:W_t1}
\end{equation}

Similarly, from \eqref{eq:V_i} and \eqref{eq:V_i_reg}, using the definition of $\mu$ and 
$\left(1 - T \gamma(T)\right)\leq e^{-T\gamma(T)}$, we have
    $V_\varepsilon(x(t_{i+1},i)) \leq e^{-\mu(t_{i+1}-t_i)} W(x(t_i,i))$,
which, by the definition of $W$, yields
\begin{align}
    W(x(t_{i+1},i)) \leq e^{-\mu(t_{i+1}-t_i)} W(x(t_i,i)). \label{eq:W_i}
\end{align}

At jumps, we have $x(t_i, i) = x(t_i,i-1)$ for all $i=1,\dots,j$.
Then, applying \eqref{eq:W_i} iteratively for $i = 1, \dots, j$, we obtain
\begin{equation}
    W(x(t,j)) \leq e^{-\mu(t - t_1)}\, W(x(t_1, 0)).   \label{eq:W_allj}
\end{equation}
Substituting~\eqref{eq:W_t1} in \eqref{eq:W_allj}, and using $\Delta_0 \leq T$, we obtain 
\begin{align}
    W(x(t,j))
    &\leq e^{-\mu(t-\Delta_0 )} \left(\hat\gamma(T)\, W(x_0) + \hat\alpha(\Delta_0 )\right) \notag\\
    &\leq e^{-\mu t} \, e^{\mu T} \hat\gamma(T) \left(W(x_0) + \alpha(\Delta_0 )\right),
    \label{eq:W_final}
\end{align}
where $\alpha(\Delta_0 ) \coloneqq \hat\alpha(\Delta_0 )/\hat\gamma(T) \in \K_\infty$.
By item~\ref{thm:time_reg:dwell}, we have $t \geq (j-1)T$ for all $(t,j) \in \dom\phi$, which implies $t + j \leq t \left( 1 + \frac{1}{T} \right) + 1$, and consequently
    $t \geq \frac{T}{T+1} (t+j-1)$.
Substituting in \eqref{eq:W_final}, we obtain
\begin{align}
    W(x(t,j)) &\leq e^{-\mu\frac{T}{T+1} (t+j)} e^{\mu\frac{T}{T+1}} e^{\mu T} \hat\gamma(T)  \left(W(x_0) + \alpha(\Delta_0 )\right) \nonumber\\
    &= e^{-\hat\mu (t+j)} C_T \left(W(x_0) + \alpha(\Delta_0 )\right),
    \label{eq:W_hybrid_time}
\end{align}
where $\hat\mu\coloneqq \mu\frac{T}{T+1}$ and $C_T\coloneqq e^{\hat\mu + \mu T} \hat\gamma(T)$.

Since $W$ is positive definite and radially unbounded with respect to $\setA_{T}$, there exist $\underline\alpha, \overline\alpha \in \K_\infty$ such that
\begin{equation}
    \underline\alpha\!\left(\norm{z}_{\setA_{T}}\right)
    \leq W(x)
    \leq \overline\alpha\!\left(\norm{z}_{\setA_{T}}\right),
    \label{eq:W_sandwich}
\end{equation}
for all $z = (x,\sigma,\tau) \in \R^n\times\N_N\times[0,2T]$.

Applying \eqref{eq:W_sandwich} to \eqref{eq:W_final}, it follows that
\begin{align}
    &\norm{\phi(t,j)}_{\setA_{T}}  \leq \underline\alpha^{-1}\! \left(e^{-\hat\mu (t+j)} C_T  \left(\overline\alpha\!\left(\norm{\phi(0,0)}_{\setA_{T}}\right)
        \!+ \alpha(\Delta_0 )\right)\!\right) \nonumber\\
    &\, \leq \underline\alpha^{-1}\! \left(e^{-\hat\mu (t+j)} C_T  \tilde\alpha\!\left(\norm{\phi(0,0)}_{\setA_{T}}
        \!+ \alpha(\Delta_0 )\right)\!\right),
    \label{eq:KL_final}
\end{align}
where $\tilde\alpha(r)\coloneqq \overline\alpha(r) + r \in\K_\infty$.
Defining $\beta\in\KL$ as $\beta(r,s)\coloneqq \underline\alpha^{-1}\left(e^{-\hat\mu s} C_T \tilde\alpha(r) \right)$ concludes the proof. \hfill $\blacksquare$


\section{Proof or Lemma~\ref{lem:MRRA}}
\label{app:MRRA}

Define the matrix 
$ \hat R\coloneqq \begin{bsmallmatrix} \phantom{-}R\\ -R \end{bsmallmatrix} $.
Therefore, $\Psi_{R,\infty}(x) = \norm{Rx}_\infty=\max_{i\in\N_{2r}}\{\hat R_{i:} x\}$.

\medskip
\noindent
\ref{lem:1:lyap}) $\Rightarrow$ \ref{lem:1:MRRA}):
Since $R$ is full column rank, we have
\begin{equation}
    \beta \coloneqq - \max_{\{x:\Psi_{R,\infty}(x) = 1\}}D^+\Psi_{R,\infty}(x) > 0. \label{eq:beta}
\end{equation}
Then, by~\cite[Theorem 4.33]{blanchini2015settheoretic}, there exists a Metzler matrix $H\in\R^{2r\times2r}$ such that $H\hat R = \hat R A$, and $H \one \preceq  - \beta \one$.
Writing $H$ as a block matrix and using the definition of $\hat R$, we can rewrite 
\begin{align}
    \begin{bmatrix}
        H^{11} & H^{12} \\ H^{21} & H^{22}
    \end{bmatrix}
    \begin{bmatrix}
        R \\ -R
    \end{bmatrix} = 
    \begin{bmatrix}
        R\\ -R
    \end{bmatrix}A,
    \label{eq:lem:1:HRRA_block}
\end{align}
where $H^{11}, H^{12}, H^{21}, H^{22}\in\R^{r\times r}$ are the submatrices of $H$.
Define $\Gamma \coloneqq H^{11}-H^{12}$. From the first block row of \eqref{eq:lem:1:HRRA_block}, we obtain $\Gamma R = R A$.
Since $H$ is Metzler, $H_{i,k}\geq0$ for all $i\neq k$, and thus $|{\bar H^{11}_{i,k} - \bar H^{12}_{i,k}}| \leq \bar H^{11}_{i,k} + \bar H^{12}_{i,k}$ for all $k \neq i$. 
From the definitions of $\Gamma$ and $\mu_\infty$ \cite{desoer1972measure}, we have
\begin{align}
    \mu_\infty(\Gamma) &= \max_{i\in\N_r}\Big(H_{i,i}^{11} - H_{i,i}^{12} + \sum_{k\in\N_r,\, k\neq i} \abs{H_{i,k}^{11} - H_{i,k}^{12}} \Big) \notag\\
    &\leq \max_{i\in\N_r} \Big(H_{i,i}^{11} + H_{i,i}^{12} + \sum_{k\in\N_r,\, k\neq i} \left(H_{i,k}^{11} + H_{i,k}^{12}\right) \Big) \notag\\
    &= \max_{i\in\N_r} \sum_{k=1}^{2r} H_{i,k} = \max_{i\in\N_r} \left[H \one \right]_i  \leq -\beta < 0, \label{eq:lem:1:beta}
\end{align}
where \eqref{eq:lem:1:beta} is a consequence of $H \one \preceq  - \beta \one$ and \eqref{eq:beta}.

\medskip
\noindent
\ref{lem:1:lyap}) $\Leftarrow$ \ref{lem:1:MRRA}):
Let $P \in \R^{r\times r}$ be the matrix with entries 
$P_{i,k} = \Gamma_{i,k}$ if $i = k$, and $P_{i,k} = \max\{0, \, \Gamma_{i,k}\}$ if $i \neq k$.
Define $N \coloneqq P - \Gamma$, so that $N_{i,k}=0$ for $i=k$, and
$N_{i,k} = \max\{0,-\Gamma_{i,k}\}$ for $i\neq k$. 
The matrix $H \coloneqq\begin{bsmallmatrix}  P & N \\ N & P \end{bsmallmatrix}$ is Metzler and satisfies $H\hat R = \hat R A$ since $P - N = \Gamma$ and by hypothesis $\Gamma R = RA$.
From the definition of $\mu_\infty$ \cite{desoer1972measure}, for all $i\in\{1,\dots,2r\}$ we have
\begin{equation}
    \sum_{k=1}^{2r} H_{i,k} \! = \! \sum_{k=1}^{r} P_{i,k} \! + \! N_{i,k} \! =\! \Gamma_{i,i} \!+\! \sum_{\substack{k=1\\ k\neq i}}^{r} \abs{\Gamma_{i,k}} \!\leq\! \mu_\infty(\Gamma) \label{eq:H_pev}
\end{equation}
Defining $\beta\coloneqq-\mu_\infty(\Gamma)>0$, from \eqref{eq:H_pev} we have $H \one \preceq -\beta \one$.
By~\cite[Theorem 4.33]{blanchini2015settheoretic} the result follows. \hfill $\blacksquare$


\bibliographystyle{IEEEtran}
\bibliography{biblio}

@book{blanchini2015settheoretic,
author = {Blanchini, F. and Miani, S.},
title = {Set-Theoretic Methods in Control},
publisher = {Birkh\"{a}user},
year = {2015},
address = {Cham, Switzerland},
}

@book{liberzon2003switched,
author = {Liberzon, D.},
title = {Switching in Systems and Control},
publisher = {Birkh\"{a}user},
year = {2003},
address = {Boston, MA},
}

@book{khalil2002nonlinear,
  title={Nonlinear systems},
  author={Khalil, Hassan K and Grizzle, Jessy W},
  volume={3},
  year={2002},
  publisher={Prentice hall Upper Saddle River, NJ}
}

@book{horn2012matrix,
  title={Matrix Analysis},
  author={Horn, Roger A and Johnson, Charles R},
  year={2012},
  edition={2},
  publisher={Cambridge University Press}
}

@book{goebel2012hybrid,
 author = {R. Goebel and R. G. Sanfelice and A. R. Teel},
 title = {Hybrid Dynamical Systems: Modeling, Stability, and Robustness},
 publisher = {Princeton University Press},
 year = {2012},
 address = {Princeton, NJ}
}

@book{lindqvist2019notes,
  title={Notes on the stationary p-Laplace equation},
  author={Lindqvist, Peter},
  year={2019},
  publisher={Springer}
}

@article{soderlind2006logarithmic,
  title={The logarithmic norm. History and modern theory},
  author={S{\"o}derlind, Gustaf},
  journal={BIT Numerical Mathematics},
  volume={46},
  number={3},
  pages={631--652},
  year={2006},
  publisher={Springer}
}

@article{perov2017spectral,
  title={On the spectral abscissa and the logarithmic norm},
  author={Perov, Anatoliy Ivanovich and Kostrub, Irina Dmitrievna},
  journal={Mathematical Notes},
  volume={101},
  number={3},
  pages={677--687},
  year={2017},
  publisher={Springer}
}

@article{cinto2025switching,
  title={Switching-controlled invariance of polytopic sets for switched affine systems with dwell time},
  author={Cinto, Felipe and Vallarella, Alexis J and Russo, Antonio and Haimovich, Hernan},
  journal={SIAM Journal on Control and Optimization},
  volume={63},
  number={5},
  pages={3167--3188},
  year={2025},
  publisher={SIAM}
}

@article{albea2019practical,
  title={Practical stabilization of switched affine systems with dwell-time guarantees},
  author={Albea, C. and Garcia, G. and Hadjeras, S. and Heemels, W. P. M. H. and Zaccarian, L.},
  journal={IEEE Transactions on Automatic Control},
  volume={64},
  number={11},
  pages={4811--4817},
  year={2019},
  publisher={IEEE}
}

@article{blanchini1995nonquadratic,
  title={Nonquadratic {L}yapunov functions for robust control},
  author={Blanchini, F.},
  journal={Automatica},
  volume={31},
  number={3},
  pages={451--461},
  year={1995},
  publisher={Elsevier}
}

@article{geromel2006stabilitystabilization,
  title={Stability and Stabilization of Continuous‐Time Switched Linear Systems},
  author={Geromel, J. C. and  Colaneri, P.},
  journal={SIAM Journal on Control and Optimization},
  volume={45},
  number={5},
  pages={1915--1930},
  year={2006},
  publisher={SIAM}
}

@article{briat2015convex,
    author = {C. Briat},
    journal = {Systems \& Control Letters},
    title = {Convex conditions for robust stabilization of uncertain switched systems with guaranteed minimum and mode-dependent dwell-time},
    volume = {78},
    pages = {63-72},
    year = {2015},
    publisher = {IEEE}
}

@ARTICLE{antsaklis2009survey,
  author={Lin, Hai and Antsaklis, Panos J.},
  journal={IEEE Transactions on Automatic Control}, 
  title={Stability and Stabilizability of Switched Linear Systems: A Survey of Recent Results}, 
  year={2009},
  volume={54},
  number={2},
  pages={308-322},
}

@ARTICLE{cheng2005stabilization,
  author={D. Cheng and L. Guo and Y. Lin and Y. Wang},
  journal={IEEE Transactions on Automatic Control}, 
  title={Stabilization of switched linear systems}, 
  year={2005},
  volume={50},
  number={5},
  pages={661-666},
}

@ARTICLE{geromel2008feedback,
  author={Geromel, JosÉ C. and Colaneri, Patrizio and Bolzern, Paolo},
  journal={IEEE Transactions on Automatic Control}, 
  title={Dynamic Output Feedback Control of Switched Linear Systems}, 
  year={2008},
  volume={53},
  number={3},
  pages={720-733},
}

@article{prieur2014,
  author={Prieur, Christophe and Teel, Andrew R. and Zaccarian, Luca},
  journal={IEEE Transactions on Automatic Control}, 
  title={Relaxed Persistent Flow/Jump Conditions for Uniform Global Asymptotic Stability}, 
  year={2014},
  volume={59},
  number={10},
  pages={2766-2771},
  doi={10.1109/TAC.2014.2311611}}

@INPROCEEDINGS{russoACC,
  author={Russo, Antonio and Paolo Incremona, Gian and Cavallo, Alberto and Colaneri, Patrizio},
  booktitle={2022 American Control Conference (ACC)}, 
  title={State Dependent Switching Control of Affine Linear Systems With Dwell Time: Application to Power Converters}, 
  year={2022},
  volume={},
  number={},
  pages={3807-3813},
  doi={10.23919/ACC53348.2022.9867761}}

@article{loskot1998further,
  title={Further comments on ``Vector norms as Lyapunov functions for linear systems"},
  author={Loskot, K and Polanski, Andrzej and Rudnicki, Ryszard},
  journal={IEEE Transactions on Automatic Control},
  volume={43},
  number={2},
  pages={289--291},
  year={1998},
  publisher={IEEE}
}

@article{desoer1972measure,
  title={The measure of a matrix as a tool to analyze computer algorithms for circuit analysis},
  author={Desoer, Charles and Haneda, Hiromasa},
  journal={IEEE Transactions on Circuit Theory},
  volume={19},
  number={5},
  pages={480--486},
  year={1972},
  publisher={IEEE}
}

@article{polanski1995infinity,
  title={On infinity norms as {L}yapunov functions for linear systems},
  author={Polanski, Andrzej},
  journal={IEEE Transactions on Automatic Control},
  volume={40},
  number={7},
  pages={1270--1274},
  year={1995},
  publisher={IEEE}
}

@inproceedings{cinto2024polytopic,
  title={Polytopic {L}yapunov Functions are not Straightforward for Minimum Dwell-Time Switched Affine Systems},
  author={Cinto, Felipe and Vallarella, Alexis J and Russo, Antonio and Incremona, Gian Paolo and Haimovich, Hernan},
  booktitle={2024 IEEE 63rd Conference on Decision and Control (CDC)},
  pages={7098--7103},
  year={2024},
  organization={IEEE}
}

@INPROCEEDINGS{bolzern2004switchedaffine,
  author={Bolzern, P. and Spinelli, W.},
  booktitle={Proceedings of the 2004 American Control Conference}, 
  title={Quadratic stabilization of a switched affine system about a nonequilibrium point}, 
  year={2004},
  volume={5},
  number={},
  pages={3890--3895},
}

@ARTICLE{Yin2023,
	author = {Yin, Hao and Jayawardhana, Bayu and Trenn, Stephan},
	title = {Stability of switched systems with multiple equilibria: A mixed stable–unstable subsystem case},
	year = {2023},
	journal = {Systems and Control Letters},
	volume = {180},
	doi = {10.1016/j.sysconle.2023.105622},
}

@ARTICLE{Makarenkov201889,
	author = {Makarenkov, Oleg and Phung, Anthony},
	title = {Dwell time for local stability of switched affine systems with application to non-spiking neuron models},
	year = {2018},
	journal = {Applied Mathematics Letters},
	volume = {86},
	pages = {89 - 94},
	doi = {10.1016/j.aml.2018.06.026},
}

@ARTICLE{DiFerdinando20241691,
	author = {Di Ferdinando, M. and Pola, G. and Di Gennaro, S. and Pepe, P.},
	title = {On Sampled-Data Control of Nonlinear Asynchronous Switched Systems},
	year = {2024},
	journal = {IEEE Control Systems Letters},
	volume = {8},
	pages = {1691 - 1696},
	doi = {10.1109/LCSYS.2024.3416405},
}

@ARTICLE{Ghawash2025,
	author = {Ghawash, Faiq and Hovd, Morten and Schofield, Brad},
	title = {Model predictive control of switched affine systems with dwell time constraints—Efficient formulation, approximation and embedded implementation},
	year = {2025},
	journal = {European Journal of Control},
	volume = {86},
	pages = {},
	doi = {10.1016/j.ejcon.2025.101347},
}

@ARTICLE{Russo2025,
	author = {Russo, Antonio and Incremona, Gian Paolo and Colaneri, Patrizio},
	title = {Stabilization of Switched Affine Systems with Dwell-Time Constraint},
	year = {2025},
	journal = {IEEE Transactions on Automatic Control},
	pages = {},
	doi = {10.1109/TAC.2025.3634172},
}

@article{Parrilo2003,
  title = {Semidefinite programming relaxations for semialgebraic problems},
  volume = {96},
  DOI = {10.1007/s10107-003-0387-5},
  number = {2},
  journal = {Mathematical Programming},
  publisher = {Springer Science and Business Media LLC},
  author = {Parrilo,  Pablo A.},
  year = {2003},
  pages = {293–320}
}

\vfill

\end{document}